\let\ORIlabel\label
\let\ORIrefstepcounter\refstepcounter
\AddToHook{package/hyperref/before}{
   \let\label\ORIlabel 
   \let\refstepcounter\ORIrefstepcounter}
\documentclass[final]{siamart220329}

\usepackage{mathtools}
\usepackage{amsfonts,latexsym,amssymb,amscd,epsf}
\newcommand{\smallnotin}{\mathrel{\scalebox{0.55}{$\notin$}}}
\usepackage[utf8]{inputenc}
\usepackage{bbm}
\usepackage[caption=false]{subfig}
\usepackage{tikz}
\usepackage{verbatim}
\usepackage{tcolorbox}
\tcbuselibrary{skins}
\usepackage{booktabs,multirow}

\usepackage{etoolbox}
\usepackage{xcolor}
\newcommand{\ignore}[1]{}
 
\usepackage{xargs}                      
\usepackage[colorinlistoftodos,prependcaption,textsize=tiny]{todonotes}
\usepackage{hyperref}

\usepackage{booktabs}
\usepackage{setspace}

\usepackage{makecell}
\usepackage{url}
\usepackage{array}
\usepackage{multirow}
\usepackage{lscape} 
\usepackage[thinc]{esdiff}

\usepackage[super]{nth}
\usepackage{algorithm}
\usepackage{algpseudocode}

\usepackage{environ}
\algnewcommand{\LineComment}[1]{\Statex \(\triangleright\) \small #1}
\algnewcommand{\Input}[1]{\Statex \textbf{input:} #1}
\algnewcommand{\Output}[1]{\Statex \textbf{output:} #1}

\usepackage{environ}

\newcommand{\enVert}[2][-1]{
	\ensuremath{\mathinner{
			\ifthenelse{\equal{#1}{-1}}{ 
				\!\left\lVert#2\right\rVert}{}
			\ifthenelse{\equal{#1}{0}}{ 
				\lVert#2\rVert}{}
			\ifthenelse{\equal{#1}{1}}{ 
				\!\bigl\lVert#2\bigr\rVert}{}
			\ifthenelse{\equal{#1}{2}}{ 
				\!\Bigl\lVert#2\Bigr\rVert}{}
			\ifthenelse{\equal{#1}{3}}{ 
				\!\biggl\lVert#2\biggr\rVert}{}
			\ifthenelse{\equal{#1}{4}}{ 
				\!\Biggl\lVert#2\Biggr\rVert}{}
		}
}}

\DeclareMathAlphabet{\mathdutchcal}{U}{dutchcal}{m}{n}
\SetMathAlphabet{\mathdutchcal}{bold}{U}{dutchcal}{b}{n}
\DeclareMathAlphabet{\mathdutchbcal}{U}{dutchcal}{b}{n}
\theoremstyle{remark}
\newtheorem{remark}[theorem]{Remark}

\usepackage{cleveref}

\usepackage{listings}

\usepackage{dsfont}
\usepackage{bm}				
\DeclareMathAlphabet{\mymathbb}{U}{BOONDOX-ds}{m}{n}
\DeclareMathAlphabet\mathbfcal{OMS}{cmsy}{b}{n}
\newcommand{\set}[1]{\mathcal{#1}}

\newcommand{\numberset}{\mathbb}

\newcommand{\N}{\numberset{N}}

\newcommand{\E}{\numberset{E}}

\newcommand{\bigO}{\mathcal{O}}
\usepackage{tensor}
 
\usepackage{graphicx}
\usepackage{subfig}
\usepackage{caption}
\usepackage{pgfplots,relsize}
\pgfplotsset{compat=1.14}
\graphicspath{ {./figures/} }
\usepackage{tikz}
\usetikzlibrary{hobby, shapes}
\usetikzlibrary{quotes,angles,positioning}
\usetikzlibrary{calc}
\usetikzlibrary{matrix}
\usetikzlibrary{patterns}

\pgfplotsset{grid style={line width=0.05pt, gray}}
\pgfplotsset{minor grid style={gray}}
\pgfplotsset{major grid style={gray}}

\newcommand{\hypref}[2]{\hyperref[#2]{#1 \ref*{#2}}}

\usepackage[customcolors]{hf-tikz} 
\usetikzlibrary{patterns}
\usetikzlibrary{matrix,decorations.pathreplacing,calc}

\pgfkeys{tikz/mymatrixenv/.style={decoration=brace,every left delimiter/.style={xshift=3pt},every right delimiter/.style={xshift=-3pt}}}
\pgfkeys{tikz/mymatrix/.style={matrix of math nodes,left delimiter=[,right delimiter={]},inner sep=2pt,column sep=1em,row sep=0.5em,nodes={inner sep=0pt}}}
\pgfkeys{tikz/mymatrixbrace/.style={decorate,thick}}

\pgfkeys{tikz/mymatrixenv/.style={decoration={brace},every left delimiter/.style={xshift=8pt},every right delimiter/.style={xshift=-8pt}}}
\pgfkeys{tikz/mymatrix/.style={matrix of math nodes,nodes in empty cells,left delimiter={[},right delimiter={]},inner sep=1pt,outer sep=1.5pt,column sep=2pt,row sep=2pt,nodes={minimum width=20pt,minimum height=10pt,anchor=center,inner sep=0pt,outer sep=0pt}}}
\pgfkeys{tikz/mymatrixbrace/.style={decorate,thick}}

\usepackage{extramath}

\definecolor{matlabred}{rgb}{0.9047,    0.1918,    0.1988}
\definecolor{matlabblue}{rgb}{0.2941    0.5447    0.7494}
\definecolor{matlabgreen}{rgb}{	0.3718    0.7176    0.3612}
\definecolor{matlaborange}{rgb}{1.0000    0.5482    0.1000}

\usepackage[utf8]{inputenc}
\usepackage{todonotes}

\newbool{showfigures}
\setbool{showfigures}{true} %
\usepackage[symbol]{footmisc}

\renewcommand{\kron}{\otimes}
\renewcommand{\ten}[1]{\mathbfcal{#1}}

\title{A PRACTICAL MODE-PARALLEL IMPLEMENTATION OF THE (H-)TUCKER DECOMPOSITION VIA RANDOMIZATION}
\author{Martina Iannacito\thanks{Dipartimento di Matematica and (AM)$^2$,
Alma Mater Studiorum Universit\`a di Bologna,
Piazza di Porta San Donato  5, I-40127 Bologna, Italy,
{\tt \{martina.iannacito,sascha.portaro,davide.palitta\}@unibo.it}} \and Sascha Portaro\footnotemark[1] \and Davide Palitta\footnotemark[1] \and Claudio Arlandini\thanks{CINECA, Via Magnanelli 2, I-40033 Casalecchio di Reno (BO), Italy, {\tt \{c.arlandini,d.brandoni\}@cineca.it}} \and Domitilla Brandoni\footnotemark[2]
 }

\ifpdf
\hypersetup{pdftitle={randTucker}}
\fi

\begin{document}
	
	\maketitle
	
	

	\begin{abstract}
	In the last decades, tensors have emerged as the right tool to represent multidimensional data in a compact yet informative manner. Moreover, it is well-known that by performing low-rank factorizations of such tensors one is often able to effectively unveil possible hidden structure in data, mainly due to unexpected dependencies among the different variables encoded in the given tensor. However, computing these factorizations is extremely energy-consuming and memory-demanding, especially for high-dimensional tensors, namely those with a large number of modes. 

In this paper we focus on two state-of-the-art tensor decompositions: the Tucker and H-Tucker decompositions.
We propose novel numerical strategies able to perform these factorizations in a \emph{mode-parallel} fashion, that is the operations required by the algorithm along all modes are performed in parallel. This is in contrast to what is achieved by many procedures available in the literature that parallelize some of the operations along each mode, e.g., tensor-times-matrix steps, while still visiting one mode at the time in a sequential manner. Our strategies make use of cutting-edge randomization techniques comprising fiber sampling and randomized range-finding steps. {\color{black}In case of Tucker decomposition, we provide upper bounds on the expected value of the error provided sufficiently large sampling parameters have been adopted.} A panel of numerical results showcases the potential of our approach in reducing both the running time and the storage demand of {\color{black} computing Tucker and H-Tucker decompositions}. Moreover, experiments carried out in HPC environments illustrate the good scaling of our mode-parallel approach.

 	\end{abstract}
	\tableofcontents

%
%
%



\section{Introduction}
In the last years, tensors have become one of the most suitable mathematical objects for representing large datasets and their fundamental role is widely recognized. 
Tensor techniques find application in a variety of different fields including deep neural networks~\cite{Panagakis2024TMinDL}, image classification and recognition~\cite{SAVAS2007993,Vasilescu2002Tensorfaces}, computer vision~\cite{Donoser2010CompVisTD}, recommender systems~\cite{Rafailidis2013TagCluster, Nanopoulos2010MusicBox}, signal processing~\cite{Sidiropoulos2017SPML}, and many others.

The main feature of tensor techniques is the possibility of encoding data of very different nature in a single, compact tool that, on one hand, obeys 
a strict mathematical formalism and, on the other hand, has the flexibility of being represented by means of one's favorite tensor format. A non\textcolor{black}{-}complete list of tensor formats available in the literature include\textcolor{black}{s} the Tucker format~\cite{Tucker1966} and its hierarchical counterpart, the H-Tucker format~\cite{Hackbusch2009HT, Grasedyck2009}, Tensor-Train (TT)~\cite{Oseledets2011}, and the Canonical Polyadic (CP) format~\cite{Hitchcock1927cpd}, to name a few. Further details on these widely used tensor formats can be found in~\cite{BallardKolda2025}. Each format comes with its own arithmetic and related pros and cons in terms of computational complexity and interpretability features.
In this paper we focus on the Tucker and H-Tucker formats as these are among the most widely used tensor formats for compressing and revealing hidden patterns in data while preserving its multiway nature. 

Given an order-$d$ tensor $\ten{X}\in\mathbb{R}^{n_1\times\cdots\times n_d}$, representing it in the Tucker format means factorizing $\ten{X}$ as the product of $d$ factor matrices $\mat{U}_i\in\R^{n_i\times r_i}$ and an order-$d$ tensor of smaller size $\ten{S}\in\mathbb{R}^{r_1\times\cdots\times r_d}$, $r_i\leq n_i$ for all $i=1,\ldots,d$, such that 
\begin{equation}\label{eq:Tucker1}
    \ten{X}=\ten{S}\times_1\mat{U}_1\times_2\cdots\times_d\mat{U}_d,
\end{equation} 
where $\times_i$ denotes the mode-$i$ product; see section~\ref{Tensor basics}.

Each factor matrix $\mat{U}_i$ tries to capture most of the information of the column space of the corresponding tensor matricization, whose column number grows exponentially with $d$. Classical Tucker decomposition algorithms form the $d$ tensor matricizations and factorize them by SVD; this approach is known as the High-Order SVD (HOSVD)~\cite{Lathauwer2000HOSVD}.

One of the main disadvantages of the Tucker format~\eqref{eq:Tucker1} is the allocation of the core tensor $\ten{S}\in\mathbb{R}^{r_1\times\cdots\times r_d}$ whose number of entries grows exponentially with the number of modes $d$. To reduce this important storage cost, Hackbusch proposed the H-Tucker format~\cite{Hackbusch2009HT} which is able to attain memory requirements that now scale linearly with $d$.
This is achieved by relying on a binary tree structure to factorize a given tensor, forming and decomposing a tensor matricization for each tree node; see~\cite{Grasedyck2009} and also section~\ref{H-Tucker} for further details.

The computation of both the Tucker and H-Tucker decompositions requires performing a significant number of SVDs of large dimensional objects \textcolor{black}{often} 
leading to excessive computational costs and memory requirements. The drawbacks of these deterministic algorithms have motivated the search for more efficient methods and randomization-based approaches are frequently employed to decrease these computational costs, boosting algorithmic performance; see, e.g.,~\cite{Ahmadi2021ReviewRandTucker}. Many of these techniques simply amount to applying state-of-the-art randomized techniques for matrix (low-rank) approximation to each matricization of the tensor of interest. While this approach often reduces the running time of the overall procedure, it does not cut down its storage demand due to the requirement of assembling the full matricizations.

In this paper we present a new Tucker decomposition technique combining tensor fiber sampling and randomized range-finding techniques. In particular, thanks to the fiber sampling step we overcome the computational costs and memory requirements of forming any tensor matricizations, while the range-finding step improves the approximation of the tensor matricization column space. 
Avoiding the allocation of any matricization is key to attain a  practical mode-parallel algorithm for Tucker decomposition. {\color{black}In particular, while the original algorithm is embarrassingly parallelizable mode-wise, standard numerical techniques need to assemble tensor matricizations, whose number of entries is equal to the one of the original tensor. This is the biggest obstacle in achieving real mode-parallel implementations. Indeed, the computational nodes of a parallel infrastructure seldom have the resources to fully allocate a copy of the tensor of interest.} Now, thanks to our fiber sampling step, we are able to allocate only a {\color{black}limited} number of fibers on each computational node thus attaining an actual mode-parallel approach able to reduce inter-node communication costs and local memory requirements. Other contributions available in the literature proposed a parallel implementation of the Tucker decomposition; see, e.g.,~\cite{parallelTucker1,parallelTucker2,Minster2024ReKron}. However, their main focus is to increase the efficiency of Tensor-Times-Matrix (TTM) operations via parallelization. Our goal here is different. Our mode-wise approach fully unlocks the natural parallel nature of the Tucker decomposition by overcoming the need of allocating $d$ copies of the entire original tensor. Clearly, along each mode, parallel TTMs can be then performed by taking inspiration from the aforementioned works to further improve the performance and scalability of the overall method. For the H-Tucker decomposition, instead, we have not been successful \textcolor{black}{in} finding any work considering the parallelization of this algorithm.

Starting from our randomized Tucker decomposition, we introduce the same modifications in the H-Tucker decomposition. Since, in this case, the number of nodes in the decomposition tree is usually larger than $d$, the computational benefits of our randomized H-Tucker variant can be even more appreciable. 

Here is a synopsis of the paper. In section~\ref{Preliminaries} we report some \textcolor{black}{preliminary} background material like the notation we adopt throughout the paper (section~\ref{Notation}), some basic concepts and definitions for handling tensors (section~\ref{Tensor basics}), and the randomized SVD (RSVD) approach for matrices (section~\ref{RandSVD}) presented in~\cite{Halko2011rHOSVD}. Section~\ref{sec:prelim:tucker} starts with briefly recalling the Tucker decomposition, its approximation by HOSVD, and the variant of the latter algorithm equipped with RSVD in place of the standard SVD. Our novel Sub-R-HOSVD is proposed in section~\ref{sec:sr-HOSVD} whereas we derive related error bounds in section~\ref{sec:errbound}. A diverse set of numerical experiments illustrating the potential of our approach is reported in section~\ref{sec:tucker:exp}. Section~\ref{H-Tucker} follows {\color{black} a pattern similar to that} of section~\ref{sec:prelim:tucker} but for the H-Tucker decomposition. In particular, after recalling the latter decomposition and the Root-to-Leaves (RtL) algorithm in its original form, we propose our randomized variant in section~\ref{Subsampled randomized LtR-HT}. Its numerical performance is displayed in section~\ref{Numerical experiments - HT}. The paper ends with some conclusions in section~\ref{Conclusions}.


\section{Preliminaries}\label{Preliminaries}

\subsection{Notation}\label{Notation} 
In this section we introduce the notation we adopt throughout the paper.
Low\textcolor{black}{er}-case, bold letters ($\vec{x}$) will denote vectors, while capital, bold letters ($\mat{A}$) will denote matrices.
Calligraphic, bold letters ($\ten{X}$) stand for order-$d$ tensors, either in full or in compressed format, of size $(n_1,\dots, n_d)^{T}$. Calligraphic capital letters ($\set{S}$) denote sets. 
To shorten the notation, we denote by $[n]$ the set of all integers between $1$ and $n$. Additionally, we will write that 
order-$d$ tensors belong to $\R^{\vec{n}}$ where $\vec{n} = (n_1,\cdots, n_d)$. The symbol $n_{\ne k}$ denotes the product of all entries of $\vec{n}$ except the $k$th one, namely $n_{\ne k}=\prod_{i=1, i\ne k}^dn_i$. {\color{black}Given a tensor $\ten{X}$ (or a matrix $\mat{X}$) $\|\ten{X}\|_F$ ($\|\mat{X}\|_F$) denotes its Frobenius norm, namely the square root of the sum of the squares of all its elements.}

\subsection{Tensor basics}\label{Tensor basics}
In this section we \textcolor{black}{review} some of the basic operations involving tensors by setting the notation and conventions that will be largely used in the remainder of this paper.

\begin{definition}
	The matricization (or unfolding) along mode-$k$ of an order-$d$ tensor $\ten{A}$ of size $\vec{n}$ results in a matrix, $\mat{A}^{(k)}$, such that
	\[
		\mat{A}^{(k)}(i_k, j) = \ten{A}(i_1, \dots, i_d),
	\]
	{\color{black}where $j = \mathbb{L}(i_1, \dots, i_{k-1}, i_{k+1}, \dots, i_d)$ with $\mathbb{L}$ the function that associates a multilinear index to a linear one, cf\textcolor{black}{.}~\cite[Definition 2.2]{BallardKolda2025}}\footnote{\textcolor{black}{As discussed in~\cite[Section 2.1-2.2]{BallardKolda2025}, the natural or reverse lexicographic order is a common choice, employed in the numerical experiments reported in this manuscript, but other orderings exist and can be employed as well.}}.
	The size of $\mat{A}^{(k)}$ is $n_k \times n_{\ne k}$.\\
	The unfolding operation can be extended to a set of modes. Let $\set{S} = \{m_1, \dots, m_k\}\subset [d]$, the matricization of $\ten{A}$ along the modes in $\set{S}$ produces a matrix, $\mat{A}^{({\set{S}})}$, such that
	\[
	\mat{A}^{({\set{S}})}(i, j) = \ten{A}(i_1, \dots, i_d),
	\]
	where $i = \mathbb{L}(i_{m_1}, \dots, i_{m_k})$ and $j = \mathbb{L}(i_{\ell_1}, \dots, i_{\ell_h})$ with $\{\ell_1, \dots, \ell_h	
	\} = [d] - \set{S}$.
	The size of $\mat{A}^{({\set{S}})}$ is $n_{\set{S}}\times n_{\notin {\set{S}}}$ where $n_{\set{S}}= \prod_{k\in\set{S}}n_k$ and $n_{\notin {\set{S}}} = (n_1\cdots n_d)/n_{\set{S}}$.
\end{definition}

\begin{definition}
	The Tensor-Times-Matrix (TTM) product of $\ten{A}\in\R^{\vec{n}}$ of order $d$ and $\mat{U}\in\R^{m \times n_k}$ along mode $k$ is denoted by $	\ten{A} \times_k \mat{U}$,
	and it results in an order-$d$ tensor $\ten{B}$ of size $(n_1, \dots,n_{k-1}, m, n_{k+1}, \dots, n_d)$ such that
	\[
	\ten{B}(i_1, \dots, i_{k-1}, j, i_{k+1}, \dots, i_d) = \sum_{{i_k} = 1}^{n_k}\ten{A}(\vec{i})\mat{U}(j,i_k).
	\]	
\end{definition}
Note that the TTM product along mode-$k$ can be equivalently computed using the unfolding along mode $k$ as 
$\Bigl(\ten{A}\times_k \mat{U}\Bigr)^{(k)} = \mat{U}\mat{A}^{(k)}$.
Finally, when we want to write the TTM product of $\ten{A}$ along each mode $k$ with all $\mat{U}_k$'s for $k=1, \dots, d$, we will write 
$(\mat{U}_1,\dots, \mat{U}_d)\ten{A}:=\ten{A}\times_1\mat{U}_1\times_2\cdots\times_d\mat{U}_d$.

\subsection{Randomized range-finder}\label{RandSVD}
The randomized range-finder algorithm~\cite{Halko2011rHOSVD} is a foundational tool in randomized numerical linear algebra. The main feature behind the success of this scheme relies on the fact that the column space of a given matrix $\mat{A}\in\R^{m\times n}$, $m\geq n$, can often be well approximated by taking a moderate number of linear combinations of its columns where the coefficients of such combinations are chosen randomly. More precisely, given $\mat{\Omega}\in\R^{n\times (k+p)}$ where $k$ is a target rank and $p$ an oversampling parameter, we have $\text{Range}(\mat{A\Omega})\approx\text{Range}(\mat{A})$. The quality of such approximation is mainly driven by the singular value distribution of $\mat{A}$ and the nature of $\mat{\Omega}$. For instance, for a Gaussian $\mat{\Omega}$, if $\mat{Q}$ denotes the orthogonal factor of the skinny QR factorization of $\mat{A\Omega}$ and $\sigma_i$ is the $i$th singular value of $\mat{A}$, then
\begin{equation}\label{eq:error_rangefinder}
    \mathbb{E}\left[\|\mat{A}-\mat{QQ}^T\mat{A}\|_F\right]\leq \left(1+\frac{k}{p-1}\right)^{1/2}\sqrt{\sum_{i=k+1}^n\sigma_i^2};
\end{equation}
see, e.g.,~\cite{Halko2011rHOSVD}.
With the matrix $\mat{Q}$ at hand, one can then compute different quantities of interest. The most prominent one is probably the SVD-like approximation of $\mat{A}$ given by the so-called randomized SVD. This simply consists \textcolor{black}{of} computing the standard SVD of $\mat{Q}^T\mat{A}=\mat{U\Sigma V}^T$ so that $\mat{A}\approx (\mat{QU})\mat{\Sigma V}^T$; see, e.g.,~\cite{Halko2011rHOSVD}.

In this paper we employ the randomized range-finder scheme, along with a preliminary fiber sampling step, and provide error estimates on the same lines of~\eqref{eq:error_rangefinder}.

\section{The Tucker decomposition}\label{sec:prelim:tucker}

The Tucker decomposition~\cite{Tucker1966} can be seen as a compression technique for tensors which relies on the concept of multilinear rank. The multilinear rank, $\vec{r}=(r_1,\ldots,r_d)^T\in\N^{d}$ of $\ten{X}\in\R^{\vec{n}}$ of order $d$ is such that
$r_k = {\rm rank}(\mat{X}^{(k)})$ for $k = 1, \dots, d$. The  Tucker decomposition factorizes $\ten{X}$ as the product of a smaller order-$d$ tensor, $\ten{S}\in\R^{\vec{r}}$ called \emph{core tensor}, and $d$ \emph{factor matrices}, $\mat{U}_k\in\R^{n_k\times r_k}$, that is
\begin{equation}
   \label{eq:Tucker}
\ten{X} = (\mat{U}_1, \dots, \mat{U}_d)\ten{S}. 
\end{equation}
{Henceforth, by \textit{low-rank} we will indicate tensors that can be factorized as in~\eqref{eq:Tucker} in a non-trivial way, with multilinear rank smaller than the tensor size.} 

Among the several available algorithms for computing the Tucker decomposition of a tensor, two of the most widely used are the High-Order SVD (HOSVD)~\cite{Lathauwer2000HOSVD} and the Sequentially Truncated HOSVD (ST-HOSVD)~\cite{Vannieuwenhoven2012STHOSVD}. 
The HOSVD finds the Tucker factor matrix $\mat{U}_k$ of $\ten{X}$ as the left orthogonal factor of the SVD of $\mat{X}^{(k)}$, for every mode $k=1,\dots, d$. The core tensor is then computed as $\ten{S} = (\mat{U}^T_1, \dots, \mat{U}^T_d)\ten{X}$.


To further compress the information of a tensor, we can compute the Tucker approximation of $\ten{X}$ at multilinear rank $\vec{t}=(t_1,\ldots,t_d)^T$ such that $t_k< r_k$ for all $k=1,\ldots,d$.
If we look for an approximation with a given multilinear rank $\vec{t}$, the truncated SVD replaces the SVD, that is we keep only the first $t_k$ left singular vectors of $\mat{X}^{(k)}$. This variant is called Truncated HOSVD (T-HOSVD). The T-HOSVD is quasi-optimal, that is\textcolor{black}{,} if $\ten{X}^\star$ is the best approximation of $\ten{X}$ at multilinear rank $\vec{t}$ and $\widehat{\ten{X}}$ is the approximation produced by the T-HOSVD, then
\[
	\|{\ten{X} - \widehat{\ten{X}}}\|_{F} \le \sqrt{d}\norm{\ten{X} - \ten{X}^\star}_{F},
\]
where $d$ is the order of $\ten{X}$. In other words, the T-HOSVD approximation has at most distance $\sqrt{d}$ from the best approximation at multilinear rank $\vec{t}$, as proven in~\cite{Vannieuwenhoven2012STHOSVD}. 

From a complexity point of view, the $d$ SVDs dominate the computational cost of the (T-)HOSVD algorithm. More precisely, this complexity  can be estimated with 
$\mathcal{O}\bigl(\prod_{j=1}^d n_j\sum_{k=1}^d\min\{n_k, n_{\ne k}\}\bigr)$ flops. Assuming that all the $d$ modes have size $n$, the computational cost of the (T-)HOSVD reduces to $\mathcal{O}(dn^{d+1})$ flops.
%
Variants of the HOSVD have been developed to reduce these computational costs, either deterministic, e.g.\textcolor{black}{,} the ST-HOSVD~\cite{Vannieuwenhoven2012STHOSVD}\textcolor{black}{,} or randomized~\cite{Ahmadi2021ReviewRandTucker}. The ST-HOSVD alternates the computation of each Tucker factor $\mat{U}_k$ with the TTM product along the corresponding mode, resulting in a distribution of the TTM steps to form the core tensor. 
This minimal modification significantly reduces the computational costs to $\mathcal{O}\bigl(\sum_{k=1}^d n_kr_{< k} n_{>k}\min\{n_k, r_{< k}n_{>k}\}\bigr)$ where $r_{< k} = \prod_{j=1}^{k-1}r_j $ and $n_{>k} = \prod_{j=k+1}^{d}n_j$. This complexity reduces to $\mathcal{O}\bigl(\sum_{k=1}^d n^{d-k + 2}r^{k-1}\bigr)$ assuming that all components of the size and the multilinear rank are equal, and $n = \min\{n, r^{k-1}n^{d-k}\}$. Finally, recalling that $r \le n$, we can estimate the computational complexity of the SVDs in the ST-HOSVD to {\color{black}amount} to $\mathcal{O}(n^{d+1})$. On the other hand, ST-HOSVD has an intrinsic serial nature not suitable for parallelization.

A first randomized variant of the HOSVD (R-HOSVD) replaces the deterministic SVD of the tensor unfolding in the (T-)HOSVD with the randomized SVD for each mode. In~\cite{Minster2020RHOSVD}, the authors discuss possible choices for the random matrices $\{\mat{\Omega}_k\}_k$ and they prove in~\cite[Theorem 3.1]{Minster2020RHOSVD} that if $\widehat{\ten{X}}$ is the approximation of $\ten{X}$ at multilinear rank {\color{black}$\vec{r}=(r_1,\ldots,r_d)^T$} generated by the R-HOSVD with Gaussian matrices $\{\mat{\Omega}_k\}_k$ and oversampling parameter $p \ge 2$ such that $r_k + p \le \min\{n_k, n_{\ne k}\}$, then
\begin{equation}\label{eq:errorbound_RHOSVD}
    	\E \left [ \| \ten{X} - \widehat{\ten{X}} \|_F \right ] \le \biggl(d + \frac{\sum_{k=1}^d r_k}{p-1}\biggr)^{\frac{1}{2}}\norm{\ten{X} - \ten{X}^\star}_F,
\end{equation}
where $\ten{X}^\star$ is the best approximation of $\ten{X}$ at multilinear rank $\vec{r}$.
The computational cost of the $d$ randomized SVDs reduces to $\mathcal{O}(\sum_{k=1}^d r_k \prod_{j=1}^{d}n_j)$ flops, that is ${\cal O}(drn^d)$ assuming that all the multilinear rank and the mode size components are equal to $r$ and $n$, respectively. Therefore, the R-HOSVD is able to reduce the main computational cost of the overall procedure by one order of magnitude assuming $r\ll n$. On the other hand, the memory requirements of HOSVD and R-HOSVD remain the same due to the allocation of the matricizations $\mat{X}^{(k)}$.


\subsection{Subsampled randomized HOSVD}\label{sec:sr-HOSVD}
In this section we present one of the most important contribution\textcolor{black}{s} of this paper. Our goal is to further reduce the complexity of computing the Tucker decomposition with a given multilinear rank {\color{black}$\vec{r}=(r_1,\ldots,r_d)^T$} in terms of both flops and storage demand.

Let \textcolor{black}{us} assume, for the time being, that $r_k=\text{rank}(\mat{X}^{(k)})$ and no truncation is performed. In this case, the
columns of the factor $\mat{U}_k$ amount to an orthogonal basis for the column space of a very fat matrix, $\mat{X}^{(k)}\in\R^{n_k\times n_{\neq k}}$. In general, $n_k\ll n_{\neq k}$, so that the column space of $\mat{X}^{(k)}$ has dimension at most $n_k$ and, in most cases, we do not need all the $n_{\neq k}$ columns of $\mat{X}^{(k)}$ to construct a basis for $\text{Range}(\mat{X}^{(k)})$. This is the reasoning behind the subsampling step we introduce in our procedure. In place of constructing an orthonormal basis for $\text{Range}(\mat{X}^{(k)})$, we construct a basis for $\text{Range}(\mat{X}^{(k)}\mat{E})$ where $\mat{E}\in\R^{n_{\neq k}\times s_k}$ contains on its columns $s_k$ randomly chosen\footnote{{\color{black}With no repetition.}} vectors of the canonical basis of $\R^{n_{\neq k}}$ for a given $s_k\geq n_k$. Clearly, if among the $s_k$ columns of $\mat{X}^{(k)}\mat{E}$,
$\text{rank}(\mat{X}^{(k)})$ of them turn out to be linearly independent, then $\text{Range}(\mat{X}^{(k)})=\text{Range}(\mat{X}^{(k)}\mat{E})$.

The construction of $\mat{Y}_k:=\mat{X}^{(k)}\mat{E}\in\R^{n_k\times s_k}$ faces a couple of challenges. The first one is computational. In particular, we want to avoid at all the construction of any matricization of $\ten{X}$. This means that we are not allowed to form $\mat{X}^{(k)}$ and then extract some of its columns to form $\mat{Y}_k$. A careful handling of index transformations does the job so that we can directly sample fibers of $\ten{X}$ - that would correspond to columns of $\mat{X}^{(k)}$ - to construct $\mat{Y}_k$. This fiber sampling step remarkably impacts the performance of our algorithm. First, {\color{black} in contrast to what happens in HOSVD,} we avoid the extremely time-consuming operation of rearranging the entries of $\ten{X}$ into the matrix $\mat{X}^{(k)}$. Second, the explicit construction of $\mat{X}^{(k)}$ requires \textcolor{black}{copying} all the $\prod_{k=1}^dn_k$ entries of $\ten{X}$; a tremendous cost in terms of storage allocation. The construction of $\mat{Y}_k$, on the other hand, requires the storage of $s_k$ fibers of $\ten{X}$ and since $s_k$ can often be chosen much smaller than $n_{\neq k}$, storing $\mat{Y}_k$ turns out to be significantly cheaper than storing $\mat{X}^{(k)}$. This fact is crucial for achieving a real mode-parallel implementation of the Tucker decomposition. Indeed, assuming we have $d$ computational nodes at our disposal, it is very unlikely that each of them has the resources to store a full copy of $\ten{X}$, namely the matricization $\mat{X}^{(k)}$. Thanks to our subsampling step, we no longer need to assemble $\mat{X}^{(k)}$ on the $k$-th computational node but rather $\mat{Y}_k$. As shown in our numerical results in section~\ref{sec:tucker:exp}, this operation is often affordable\textcolor{black}{,} leading to the realistic possibility of running a mode-parallel implementation of the Tucker decomposition.

The second challenge we need to face when we perform a subsampling step is clearly how to choose the \emph{right} columns of $\mat{X}^{(k)}$ to form $\mat{Y}_k=\mat{X}^{(k)}\mat{E}$. Many selection strategies to sample sensible columns of matrices have been proposed in the literature; see, e.g.,~\cite{Stewart99,Cortinovisetal2020,Thurauetal,Bou2017,Drineas2008,Wang2013}. However, computing the measures providing the best column indices is very expensive, with a cost that is often comparable to the one of computing the SVD of $\mat{X}^{(k)}$. We thus refrain from employing any of these techniques and we simply pick $s_k$ indices in $[n_{\neq k}]$ uniformly at random {\color{black}with no repetition}. It is easy to construct counterexamples where this naive strategy fails, namely, in our context, $\text{Range}(\mat{X}^{(k)}\mat{E})$ does not contain any sensible information coming from $\text{Range}(\mat{X}^{(k)})$. On the other hand, in section~\ref{sec:errbound} we picture the right theoretical frame where our \textcolor{black}{subsampling} step succeeds\textcolor{black}{,} and in section~\ref{sec:exp:realworld} \textcolor{black}{and in the supplementary materials} we report a series of real-world examples where the assumptions we draw are indeed satisfied.

Once we have $\mat{Y}_k=\mat{X}^{(k)}\mat{E}$ at hand, we compute an approximation to $\text{Range}(\mat{Y}_k)$ by the range-finder algorithm recalled in section~\ref{RandSVD}. Indeed, our initial goal was to construct a sensible subspace of dimension $r_k\leq n_k$ able to well approximate $\text{Range}(\mat{X}^{(k)})$. The $s_k\geq n_k$ columns $\mat{Y}_k$ \textcolor{black}{may not all be needed} to this end. Therefore, we rather draw a Gaussian matrix $\mat{\Omega}_k\in\R^{s_k\times (r_k+p)}$ for a given oversampling parameter $p$, and set {\color{black}$\mat{U}_k=\mat{Q}_k\mat{W}_k$ where $\mat{Q}_k\in\mathbb{R}^{n_k\times (r_k+p)}$ is the orthogonal factor of the skinny QR of $\mat{Y}_k\mat{\Omega}_k$ whereas $\mat{W}_k\in\mathbb{R}^{(r_k+p)\times r_k}$ collects the dominant $r_k$ left singular vectors of $\mat{Q}_k^T\mat{Y}_k$}. Notice that the construction of a much smaller Gaussian matrix in our setting -- $\mat{\Omega}_k$ is $s_k\times (r_k+p)$ instead of $n_{\ne k}\times (r_k+p)$ as in R-HOSVD -- is another important feature of our scheme. Indeed, even if it is linear in the dimension of $\mat{\Omega}_k$, the cost of drawing a Gaussian matrix should not get overlooked, especially in a tensor setting where $n_{\ne k}$ grows exponentially with the number of modes $d$.

The reasoning behind our overall approach can thus be summarized as follows\textcolor{black}{:}
$$\text{Range}(\mat{U}_k)\underset{\substack{\\\text{Range Finder}}}{\approx} \text{Range}(\mat{X}^{(k)}\mat{E}) \underset{\substack{\\\text{Subsampling}}}{\approx} \text{Range}(\mat{X}^{(k)}),$$
and the related routine, named the Sub-R-HOSVD, is shown in Algorithm~\ref{alg:Sub-R-HOSVD}. 
Its computational cost is $\mathcal{O}(\sum_{k=1}^d n_ks_kr_k+n_kr_k^2)$, which, assuming $r_k=r, s_k=s$, and $n_k=n$ for all modes, simplifies to ${\cal O}(d ( n s r +n r^2))$ which leads to significant computational gains for {\color{black}$s\ll n^{d-1}$}. \textcolor{black}{Notice that the complexity of the core tensor computation is omitted, since this cost remains the same for all the deterministic and randomized variants of the HOSVD.} The main memory usage of Algorithm~\ref{alg:Sub-R-HOSVD} is given by the storage of the matrix $\mat{Y}_k$ that amounts to $s_k$ vectors of length $n_k$, again way less than $\prod_{i=1}^dn_i$ for $s_k\ll n_{\neq k}$.

\begin{algorithm}[t!]
	\centering
	\caption{$[\ten{S}\,, \{\mat{U}_k\}]$ = Sub-R-HOSVD($\ten{X}$, $\vec{r}$, $p$, $\vec{s}$)}\label{alg:Sub-R-HOSVD}
	\begin{algorithmic}[1]
		\Input{$\ten{X}$ an order-$d$ tensor of size $\vec{n}$, and a prescribed multilinear rank $\vec{r}$, oversampling parameter $p \ge 4$ and multilinear sampling parameter $\vec{s}$}
		\Output{\small$\ten{S}$ the core order-$d$ tensor of size $\vec{r}$, and the factor matrices $\{\mat{U}_k\}_k$}
		\For{$k = 1, \ldots, d$}
		\State Randomly pick $s_k$ mode-$k$ fibers of $\ten{X}$ and store them into the matrix $\mat{Y}_k \in \R^{n_k \times s_k}$
        \State Generate a Gaussian sketch matrix $\mat{\Omega}_k \in \R^{s_k \times (r_k+p)}$
		\State Compute an orthonormal basis {\color{black}$\mat{Q}_k \in \R^{n_k \times (r_k+p)}$} of $\text{Range}(\mat{Y}_k \mat{\Omega}_k)$
        \State {\color{black}Compute the dominant $r_k$ left singular vectors of $\mat{Q}_k^T\mat{Y}_k$ and collect them in $\mat{W}_k$}
        \State {\color{black}Set $\mat{U}_k=\mat{Q}_k\mat{W}_k \in \R^{n_k \times r_k}$}
		\EndFor
		\State Compute $\ten{S}$ as $(\mat{U}^T_1, \dots, \mat{U}^T_d)\ten{X}$\label{alg:subRHOSVD_computecore}
	\end{algorithmic}
\end{algorithm}

\subsection{Error bounds}
\label{sec:errbound}
In this section we show that the Tucker decomposition computed by Algorithm~\ref{alg:Sub-R-HOSVD} fulfills error bounds similar to the one in~\eqref{eq:errorbound_RHOSVD}. To this end, we first work mode by mode, showing that the subsampling of the column of $\mat{X}^{(k)}$ followed by a range\textcolor{black}{-}finding step leads to the construction of good approximations of $\text{Range}(\mat{X}^{(k)})$. Then, we put these error bounds together to estimate {\color{black}$\E[\|\ten{X}- \ten{\widehat X}\|_F]$} where $ \ten{\widehat X}$ is the approximation given by Algorithm~\ref{alg:Sub-R-HOSVD} whereas $\ten{X}$ is the original tensor.

By taking inspiration from the results in~\cite{palitta2024}, our analysis starts with studying the effects of the subsampling step on a fat matrix $\mat{X}\in\R^{n\times m}$ with $m\gg n$. This matrix plays the role of the matricizations $\mat{X}^{(k)}$. 
We first partition $\mat{X}$ as follows
\begin{equation}\label{partition}
    \begin{array}{cccc}
     & \displaystyle r \quad n-r& m   \\
     \mat{X} = \mat{U}&   {\begin{bmatrix} \mat{\Sigma}_{1} \\ & \mat{\Sigma}_{2} \end{bmatrix}} & {\begin{bmatrix} \mat{V}_{1}^T \\ \mat{V}_{2}^T \end{bmatrix}}   & {\begin{matrix} r  \\ n-r \end{matrix}}
\end{array}
\end{equation}
where $r$ will be the $k$-th component $r_k$ of the target multilinear rank $\vec{r}$ in Algorithm~\ref{alg:Sub-R-HOSVD}.

In the next lemma we present some criteria for choosing the subsampling parameter $s$ such that $\mat{XE}$ preserves some good properties, where $\mat{E}\in\R^{m\times s}$ denotes the sampling matrix. These criteria are given in terms of the \emph{coherence} of the row space of $\mat{X}$. In general, given a subspace $\mathcal{V} \subseteq \R^m$ and an orthogonal projector $\mat{P}_{\mathcal{V}}=\mat{VV}^T$ onto $\mathcal{V}$, the coherence of $\mathcal{V}$ is defined as
$$
\mu(\mathcal{V}):=\max_{i=1,\ldots,m}\|\mat{P}_{\mathcal{V}}\vec{e}_i\|_2^2=\max_{i=1,\ldots,m}\|\mat{V}^T\vec{e}_i\|_2^2,
$$
where $\vec{e}_i\in\R^m$ denotes the $i$-th canonical basis vector of $\R^m$. Following the partition in~\eqref{partition}, we are interested in the coherence of $\text{Range}(\mat{V}_1)$ and $\text{Range}(\mat{V}_2)$ that, for the sake of brevity and with abuse of notation, we will denote as $\mu(\mat{V}_1)$ and $\mu(\mat{V}_2)$, respectively.

\begin{lemma}\label{lemma_fullrank}
    Let $\mat{X}=\mat{U\Sigma V}^T$ with $\mat{V}=[\mat{V}_1,\mat{V}_2]$, $\mat{V}_1\in \R^{m \times r}$, $\mat{V}_2\in\R^{m \times (n-r)}$ as in~\eqref{partition}. Define the quantities $M_i := m\cdot\mu(\mat{V}_i)$, $i=1,2$. Select the sample size $s \ge \max \{\alpha_1 M_1 \log (r),\alpha_2 M_2 \log (n-r) \}$, with $\alpha_1, \alpha_2 \ge 0$. Then, for $\delta \in [0,1)$ and $\eta \in [0,m/s-1]$, it holds
    \begin{equation}\label{eq:Lemma3.1}
        \sigma_r(\mat{V}_1^T \mat{E}) \ge \sqrt{\frac{(1-\delta) s}{m}} \qquad \text{and} \qquad \sigma_1(\mat{V}_2^T \mat{E}) \leq \sqrt{\frac{(1+\eta) s}{m}},
    \end{equation}
    with failure probability at most $r \cdot \left ( \frac{\mathtt e^{-\delta}}{(1-\delta)^{1-\delta}} \right )^{\alpha_1 \log (r)} + (n-r) \cdot \left ( \frac{\mathtt e^{\eta}}{(1+\eta)^{1+\eta}} \right )^{\alpha_2 \log (n-r)}$. 
\end{lemma}

\begin{proof}
The first part directly comes from applying~\cite[Lemma 3.4]{Tropp2011} twice: in the first case, we consider only the lower bound for the {\color{black}$r$}-th singular value of $\mat{V}_1^T \mat{E}$, while, in the second case, we use only the upper bound for the first singular value of $\mat{V}_2^T \mat{E}$. The overall failure probability of these two events follows from a simple union bound. Moreover, the upper bound on $\eta$, namely $\eta\leq m/s-1$ comes from simply noticing that $\sigma_1(\mat{V}_2^T \mat{E} )\leq 1$ always holds true, regardless of $s$.
\end{proof}

Lemma~\ref{lemma_fullrank} shows that if the subsampling parameter $s$ is sufficiently large, then, with high probability, the subsampled right singular matrix $\mat{V}_1^T\mat{E}$ is full rank; this is crucial for showing the effectiveness of the range finding step, as shown next. 
Moreover, it also provides the following bound
\begin{equation}\label{bound_frobenius_norm_UT1_E}
    \| (\mat{V}_1^T \mat{E})^{\dagger} \|_F^2 = \sum_{i=1}^r\frac{1}{\sigma_i(\mat{V}_1^T \mat{E})^2}\leq \frac{r}{\sigma_r(\mat{V}_1^T \mat{E})^2} \leq \frac{r \cdot m}{(1-\delta) s},
\end{equation}
which holds with failure probability at most $r \cdot \left ( \frac{\mathtt e^{-\delta}}{(1-\delta)^{1-\delta}} \right )^{\alpha_1 \log (r)}$, if $s \ge \alpha_1 M_1 \log (r)$.


Before continuing, let us discuss first the meaning of the subsampling parameter $s$ being \emph{sufficiently large}. As shown in Lemma~\ref{lemma_fullrank}, this must be chosen in terms of coherence as it is customary in subsampling techniques; see, e.g.,~\cite{Ipsen_coherence}. {\color{black} For a matrix $\mat{X}\in\mathbb{R}^{n\times m}$ with low coherence, we can select moderate sampling parameters $s$, much smaller than $m$, without jeopardizing the likelihood of~\eqref{eq:Lemma3.1}. On the other hand,}
we must mention that it is easy to construct examples where a subsampling  step fails, namely $\mat{XE}$ retains information of $\mat{X}\in\R^{n\times m}$ only for extremely large $s\approx m$. This is the case, e.g\textcolor{black}{.}, when $\mu(\mat{V})=1$. In these scenarios, the subsampling step does not lead to any computational gain. Our selection of $s$, however, is in terms of $\mu(\mat{V}_1)$ and $\mu(\mat{V}_2)$ which are in general smaller than $\mu(\mat{V})$.
Moreover, in certain practical situations one could expect $\mu(\mat{V})$ to be small by, e.g., having prior knowledge about the application setting of interest. Indeed, $\mu(\mat{V})$ may have some proper physical meaning. For instance, when clustering objects from partial observations, the coherence is a function of the minimum cluster size~\cite{Coherence1}. Similarly, in recovering spectrally sparse signals, low coherence means that the supporting frequencies are spread out; a natural assumption in this framework~\cite{Coherence2}. We also want to stress that the concept of coherence is not strictly related to that of sparsity. In particular, a very sparse matrix does not \textcolor{black}{necessarily need} to have a large coherence. For instance, in~\cite{palitta2024} a panel of results shows that the sampling step does not lead to any accuracy deterioration even in \textcolor{black}{the} case of extremely sparse matrices. Similarly,
in section~\ref{sec:tucker:exp} we show that for tensors collecting application-driven data, performing a subsampling step does not cause any loss of accuracy, also for modest subsampling parameters $s$. This suggests that we can indeed expect {\color{black}many} real-world datasets to have low coherence.

We now proceed with our error analysis presenting the following lemma.

\begin{lemma}\label{lemma_fullrankOmega}
Consider the same assumptions and notation of Lemma~\ref{lemma_fullrank}. Moreover, let $\mat{\Omega}\in\R^{s\times(r+p)}$ be a Gaussian matrix.
Then, $\mat{V}_1^T \mat{E} \mat{\Omega}$ has full rank with probability $1-r \cdot \left ( \frac{\mathtt e^{-\delta}}{(1-\delta)^{1-\delta}} \right )^{\alpha_1 \log (r)}$.
\end{lemma}

\begin{proof}
Under the assumptions of Lemma~\ref{lemma_fullrank}, $\mat{V}_1^T \mat{E}\in\R^{r\times s}$ is full rank with high probability so that if $\mat{W} \mat{S} \mat{Z}^T$ denotes its SVD, then the matrix $\mat{Z} \in \R^{s \times r}$ has full rank with the same probability, i.e. $1- r \cdot \left ( \frac{\mathtt e^{-\delta}}{(1-\delta)^{1-\delta}} \right )^{\alpha_1 \log (r)}$.  
Since $\mat{Z}$ has orthogonal columns, $\mat{Z}^T \mat{\Omega}$ is Gaussian and, consequently, the matrix $  \mat{V}_1^T \mat{E}\mat{\Omega}= \mat{W}  \mat{S} (\mat{Z}^T \mat{\Omega})$ is full rank with the same probability of $\mat{V}_1^T \mat{E}$ being full rank.
\end{proof}

With Lemma~\ref{lemma_fullrank} and~\ref{lemma_fullrankOmega}, we are now able to show that $\text{Range}(\mat{XE\Omega})\approx \text{Range}(\mat{X})$ for a fat matrix $\mat{X}$. This means that the operations we perform in each loop of Algorithm~\ref{alg:Sub-R-HOSVD} construct a matrix $\mat{U}_k$ whose columns span a good approximation of $\text{Range}(\mat{X}^{(k)})$.

For the sake of simplicity in the presentation of the results that follow, we define the following event
\begin{equation}\label{event_B}
    B = \left \{ \sigma_r(\mat{V}_1^T\mat{E}) \ge \sqrt{\frac{(1-\delta) s}{m}} \qquad \text{and} \qquad \sigma_1(\mat{V}_2^T\mat{E}) \leq \sqrt{\frac{(1+\eta) s}{m}} \right \},
\end{equation}
which, thanks to Lemma~\ref{lemma_fullrank}, is very likely for
 $s$ large enough.

\begin{theorem}\label{expected_error_under_B}
Suppose that $\mat{X}$ is a real $n \times m$ matrix, $n \leq m$, partitioned as in~\eqref{partition} where we fix a target rank $r \ge 2$, an oversampling parameter $p \ge 4$, such that $r + p \leq n$, and assume $\text{rank}(\mat{X}) \geq r$. Choose a subsampling parameter $s$ as in Lemma \ref{lemma_fullrank}. {\color{black}Define $\mat{U}=\mat{QW}\in\mathbb{R}^{n\times r}$ such that the columns of $\mat{Q}\in\R^{n\times ({\color{black}r}+p)}$
amount to an orthonormal basis of $\text{Range}(\mat{XE\Omega})$ with $\mat{E}\in\R^{m\times s}$ sampling matrix and $\mat{\Omega}\in\R^{s\times ({\color{black}r}+p)}$ Gaussian, whereas $\mat{W}\in\mathbb{R}^{(r+p)\times r}$ collects the dominant $r$ left singular vectors of $\mat{Q}^T\mat{X}\mat{E}$.} Then it holds
    \begin{equation}
        \E_B \left [ \| \mat{X} -{\color{black}\mat{UU}^T} \mat{X} \|_F \right ] \leq \left ( 1 +  3^{\frac{3}{2}}\frac{r \cdot (1+\eta)(r+p)}{(1-\delta)(p+1)} \right )^{\frac{1}{2}} \| \mat{\Sigma}_2 \|_F,
    \end{equation}
    where $\E_B$ is the conditional expectation under the event $B$ in \eqref{event_B}.
\end{theorem}
\begin{proof}
We first observe that $\mat{H}_{1}=\mat{V}_1^T \mat{E} \mat{\Omega}$ has full rank conditioned to the event $B$ as shown in Lemma~\ref{lemma_fullrankOmega}. Thus, the H\"older inequality, along with~{\color{black}\cite[Theorem 9]{ZhangetAL2018}}, implies that
    \begin{equation}
        \E_B \left [ \| \mat{X} - {\color{black}\mat{UU}^T} \mat{X} \|_F \right ] \leq \left ( \E_B \left [ \| \mat{X} - {\color{black}\mat{UU}^T} \mat{X} \|_F^2 \right ] \right )^{\frac{1}{2}} \leq \left ( \| \mat{\Sigma}_2 \|_F^2 + \E_B \left [ \| \mat{\Sigma}_2 \mat{H}_{2} \mat{H}_{1}^\dagger \|_F^2 \right ] \right )^{\frac{1}{2}},
    \end{equation}
    where $\mat{H}_{2}=\mat{V}_2^T \mat{E} \mat{\Omega}$.    
    Now, by the Cauchy-Schwartz inequality
    {\small
    \begin{equation*}
        \displaystyle{\E_B \left [ \| \mat{\Sigma}_2 \mat{H}_{2} \mat{H}_{1}^\dagger \|_F^2 \right ]  \leq \E_B \left [ \| \mat{\Sigma}_2 \mat{H}_{2} \|_F^4 \right ]^{\frac{1}{2}} \E_B \left [ \| \mat{H}_{1}^\dagger \|_F^4 \right ]^{\frac{1}{2}} = \E_B \left [ \| \mat{\Sigma}_2 \mat{V}_2^T \mat{E} \mat{\Omega} \|_F^4 \right ]^{\frac{1}{2}} \E_B \left [ \| \left ( \mat{V}_1^T \mat{E} \mat{\Omega} \right )^\dagger \|_F^4 \right ]^{\frac{1}{2}}}.
    \end{equation*}
    }
    {\color{black}Since $\mat{E}$ and $\mat{\Omega}$ are independent, applying Equation~(A.1d) of Lemma~A.1 in~\cite{Persson2025}, together with the fact that the Frobenius norm dominates the Schatten-4 norm, yields}
    \begin{align*}
        \E_B \left [ \| \mat{\Sigma}_2 \mat{V}_2^T \mat{E} \mat{\Omega} \|_F^4 \right ]^{\frac{1}{2}} &= \E \left [ \E \left [ \| \mat{\Sigma}_2 \mat{V}_2^T \mat{E} \mat{\Omega} \|_F^4 \; \big | \; \mat{E} \right ] \; \big | \; B \right ]^{\frac{1}{2}} \leq \E \left [ 3 \| \mat{\Sigma}_2 \mat{V}_2^T \mat{E} \|_F^4 \| \mat{I}_{r+p} \|_F^4 \; \big | \; B \right ]^{\frac{1}{2}} \\
        &\leq \sqrt{3} (r+p) \| \mat{\Sigma}_2 \|_F^2 \E \left [ \| \mat{V}_2^T \mat{E} \|_2^4 \; \big | \; B \right ]^{\frac{1}{2}} \leq \sqrt{3} (r+p) \frac{(1+\eta) s}{m} \| \mat{\Sigma}_2 \|_F^2,
    \end{align*}
    where we used the definition \eqref{event_B} of the event $B$ in the last inequality. Moreover, {\color{black}since $\E \left [ \| \left ( \mat{V}_1^T \mat{E} \mat{\Omega} \right )^\dagger \|_F^4 \; \big | \; \mat{E} \right ] \leq \frac{3}{p+1} \| (\mat{V}_1^T \mat{E})^{\dagger} \|_F^4$ (see the proof of~\cite[Proposition B.1]{palitta2024} with $p=2$), we have}

    \begin{align*}
        \E_B \left [ \| \left ( \mat{V}_1^T \mat{E} \mat{\Omega} \right )^\dagger \|_F^4 \right ]^{\frac{1}{2}} &= \E \left [ \E \left [ \| \left ( \mat{V}_1^T \mat{E} \mat{\Omega} \right )^\dagger \|_F^4 \; \big | \; \mat{E} \right ] \; \big | \; B \right ]^{\frac{1}{2}} \leq \frac{3}{p+1} \E \left [ \| (\mat{V}_1^T \mat{E})^{\dagger} \|_F^4 \; \big | \; B \right ]^{\frac{1}{2}} \\
        &\leq \frac{3 r \cdot m}{(p+1)(1-\delta) s},
    \end{align*}
    where we used the definition \eqref{event_B} of the event $B$ again in the last inequality. Hence, the thesis.
\end{proof}    

As already mentioned, the results of Theorem~\ref{expected_error_under_B} can be used to estimate the quality of the matrix $\mat{U}_k$ computed at the $k$-th iteration of Algorithm~\ref{alg:Sub-R-HOSVD} in approximating $\text{Range}(\mat{X}^{(k)})$. We now leverage this result to obtain a bound on the error provided by the whole approximation $\ten{\widehat X}=(\mat{U}_1,\ldots,\mat{U}_d)\ten{S}\approx \ten{X}$ computed by Algorithm~\ref{alg:Sub-R-HOSVD}.
To this end, we first define the family of events 
\begin{equation}
B_k = \left \{ \sigma_{r_k}( (\mat{V}_1^{(k)})^T\mat{E}_k) \ge \sqrt{\frac{(1-\delta_k) s_k}{n_{\ne k}}} \qquad \text{and} \qquad \sigma_1( (\mat{V}_2^{(k)})^T\mat{E}_k) \leq \sqrt{\frac{(1+\eta_k) s_k}{n_{\ne k}}} \right \},
\end{equation}
that corresponds to~\eqref{event_B} for the unfolding $\mat{X}^{(k)}$, so that $\mat{V}_1^{(k)}\in\R^{n_{\neq k}\times r_k}$ collects the first $r_k$ right singular vectors of $\mat{X}^{(k)}$ on its columns, $\mat{V}_2^{(k)}\in\R^{n_{\neq k}\times (n_k-r_k)}$ the remaining right singular vectors, and so far for all the other quantities that are now indexed by $k$. In particular, following Lemma~\ref{lemma_fullrank}, $B_k$ holds true with probability $1-r_k \cdot \left ( \frac{\mathtt e^{-\delta_k}}{(1-\delta_k)^{1-\delta_k}} \right )^{\alpha_1^{(k)} \log (r_k)} {\color{black} -} (n_k-r_k) \cdot \left ( \frac{\mathtt e^{\eta_k}}{(1+\eta_k)^{1+\eta_k}} \right )^{\alpha_2^{(k)} \log (n_k-r_k)}$ if $s_k \ge \max \{\alpha_1^{(k)} M_1^{(k)} \log (r_k),\alpha_2^{(k)} M_2^{(k)} \log (n_k-r_k) \}$, with $\alpha_1^{(k)}, \alpha_2^{(k)} \ge 0$ and $M_i^{(k)}=n_{\neq k}\cdot \mu(\mat{V}_i^{(k)})$, $i=1,2$.
{\color{black} We remind the reader once again that for tensors whose matricizations have small coherence we can select the sampling parameter $s_k$ much smaller than $n_{\neq k}$; see, e.g., section~\ref{sec:tucker:exp} where selecting $s_k$ as a tiny fraction of $n_{\neq k}$ still leads to accurate approximations.}
If now we define the event $\mathfrak{B} \vcentcolon = \cap_{k = 1}^d {B_k}$,
a simple union bound shows that $\mathfrak{B}$ holds with probability at least
{\color{black}
\begin{equation}\label{eq:probabilityBinter}
    \exp{\left ( -\sum_{k=1}^d\left( \frac{p_k}{1-p_k}\right) \right )},
\end{equation}
where $p_k \vcentcolon = r_k \cdot \left ( \frac{\mathtt e^{-\delta_k}}{(1-\delta_k)^{1-\delta_k}} \right )^{\alpha_1^{(k)} \log (r_k)} + (n_k-r_k) \cdot \left ( \frac{\mathtt e^{\eta_k}}{(1+\eta_k)^{1+\eta_k}} \right )^{\alpha_2^{(k)} \log (n_k-r_k)}$ and} provided all the entries of the multilinear index $\vec{s}=(s_1,\ldots,s_d)$ are properly chosen.

\begin{theorem}\label{theorem:errorbound_SubRHOSVD}
    If $\widehat{\ten{X}}$ is the approximation of $\ten{X}$ at multilinear rank $\vec{r}$ generated by Algorithm~\ref{alg:Sub-R-HOSVD} with Gaussian matrices $\{\mat{\Omega}_k\}_k$, oversampling parameter $p \ge 4$ such that $r_k + p \le \min\{n_k, \prod_{j \ne k}{n_j}\}$ and multilinear subsampling parameter $\vec{s}=(s_1,\ldots,s_d)^T$ such that $s_k \le n_{\neq k}$ is large enough, then for $\eta_k = \eta_k(s_k) \ge 0$ and $\delta_k = \delta_k(s_k) \in [0,1)$ it holds
    \begin{equation}\label{eq:boundTucker}
	   \E_{\mathfrak{B}} \left [ \| \ten{X} - \widehat{\ten{X}} \|_F  \right ] \le \left ( d + \sum_{k=1}^d 3^{\frac{3}{2}}\frac{r_k \cdot (1+\eta_k)(r_k+p)}{(1-\delta_k)(p+1)} \right )^{\frac{1}{2}} \| \ten{X} - \ten{X}^\star \|_F,
    \end{equation}
    where $\E_{\mathfrak{B}}$ is the conditional expectation under the previously defined event $\mathfrak{B}$ 
    which holds with high probability for properly chosen $\vec{s}$.
    \end{theorem}
\begin{proof}
    By Lemma 2.1 in~\cite{Minster2020RHOSVD},  we can write
    \begin{equation*}
        \E_{\mathfrak{B}} \left [ \| \ten{X} - \widehat{\ten{X}} \|_F^2  \right ] \le \E_{\mathfrak{B}} \left [ \sum_{k=1}^d \| \ten{X} \times_k \big ( \mat{I} - {\color{black}\mat{U}_k\mat{U}_k^T} \big ) \|_F^2  \right ] = \sum_{k=1}^d \E_{\mathfrak{B}} \left [ \| \ten{X} \times_k \big ( \mat{I} -  {\color{black}\mat{U}_k\mat{U}_k^T} \big ) \|_F^2  \right ].
    \end{equation*}
    Noticing that $ \| \ten{X} \times_k \big ( \mat{I} - {\color{black}\mat{U}_k\mat{U}_k^T} \big ) \|_F^2 =  \| \big ( \mat{I} - {\color{black}\mat{U}_k\mat{U}_k^T} \big ) \mat{X}^{(k)} \|_F^2$ and by applying Lemma~\ref{expected_error_under_B}, we have
    \begin{equation*}
        \E_{\mathfrak{B}} \left [ \| \ten{X} - \widehat{\ten{X}} \|_F^2  \right ] \le \sum_{k=1}^d \left ( 1 + 3^{\frac{3}{2}}\frac{r_k \cdot (1+\eta_k)(r_k+p)}{(1-\delta_k)(p+1)} \right )  \|\mat{\Sigma}_2^{(k)}\|_F^2.
    \end{equation*}
    Thus, by H\"older inequality
    \begin{align*}
        \E_{\mathfrak{B}} \left [ \| \ten{X} - \widehat{\ten{X}} \|_F  \right ] &\le \left ( \E_{\mathfrak{B}} \left [ \| \ten{X} - \widehat{\ten{X}} \|_F^2  \right ] \right )^\frac{1}{2} \\
        &\le \left ( \sum_{k=1}^d \left ( 1 + 3^{\frac{3}{2}}\frac{r_k \cdot (1+\eta_k)(r_k+p)}{(1-\delta_k)(p+1)} \right )  \|\mat{\Sigma}_2^{(k)}\|_F^2 \right )^\frac{1}{2} \\
        &\le \left ( d + \sum_{k=1}^d 3^{\frac{3}{2}}\frac{r_k \cdot (1+\eta_k)(r_k+p)}{(1-\delta_k)(p+1)} \right )^{\frac{1}{2}} \| \ten{X} - \ten{X}^\star \|,
    \end{align*}
    since $\|\mat{\Sigma}_2^{(k)}\|_F^2 \le \| \ten{X} - \ten{X}^\star \|^2$.
\end{proof}

{\color{black} We believe it is important to stress the role of the scalars $\delta_k$ and $\eta_k$ in~\eqref{eq:boundTucker} and their relation to the sampling parameter $s_k$. By keeping an eye on Lemma~\ref{lemma_fullrank} for the matrix case, we can state that having a sufficiently large $s_k$ allows us to select very modest values of $\delta_k$ and $\eta_k$ without compromising the likelihood of the event $\mathfrak{B}$\footnote{{\color{black}A large $s_k$ allows for large values of $\alpha_1^{(k)}$ and $\alpha_2^{(k)}$ in~\eqref{eq:probabilityBinter}.}}. If we consider the extreme case of having $\delta_k=\eta_k=0$ for all $k$, then the bound in~\eqref{eq:boundTucker} turns out to be of the same quality of the one given by R-HOSVD, namely~\eqref{eq:errorbound_RHOSVD}. Once again, the magnitude of $s_k$ must be related to the coherence of the corresponding unfolding $\mat{X}^{(k)}$. For low-coherence tensors, namely tensors whose unfoldings have all low coherence, moderate values of $s_k$ are sufficient to achieve very successful compressions. In particular, while $s_k$ must be greater than $n_k$\footnote{{\color{black}As a rule of thumb, we believe selecting $s_k\geq 10\cdot n_k$ is a sensible option.}}, the former can be selected orders of magnitude smaller than $n_{\neq k}$ as shown in the next sections.
}
\subsection{Numerical results}\label{sec:tucker:exp}
In this section we report the results of our numerical testing. We span from artificially constructed problems (section~\ref{sec:tucker:exp:synthetic}--\ref{sec:experiments:parallel}) to real-world ones (section~\ref{sec:exp:realworld}), from a sequential setting (section~\ref{sec:tucker:exp:synthetic}--\ref{sec:exp:realworld}) to a parallel computing environment (section~\ref{sec:experiments:parallel}). Further numerical experiments can be found in the supplementary material. We compare the performance of Algorithm~\ref{alg:Sub-R-HOSVD} against that achieved by several state-of-the-art techniques. In particular,
\begin{description}
    \item[High Order Singular Value Decomposition (HOSVD)] is a deterministic\textcolor{black}{,} widely used algorithm to compute the Tucker decomposition, briefly described in Section~\ref{sec:prelim:tucker}. We implement it in \texttt{Python} using the \texttt{NumPy} library for the classical matrix operations and the \texttt{TensorLy} library to perform the tensor matricization and the TTM product.
    \item[Sequentially Truncated-HOSVD (ST-HOSVD)] is a deterministic\textcolor{black}{,} widely used algorithm to compute the Tucker decomposition, with lower computational costs than the HOSVD, as described in Section~\ref{sec:prelim:tucker}. We implement it in \texttt{Python} using the \texttt{NumPy} library for the classical matrix operations and the \texttt{TensorLy} library to perform the tensor matricization and the TTM product.
    \item[Randomized-HOSVD (R-HOSVD){\rm,}] originally proposed in~\cite{Halko2011rHOSVD}, is a randomized algorithm with the same structure \textcolor{black}{as} the HOSVD. The key difference is that the deterministic SVD of the tensor unfoldings is replaced by the randomized SVD for each mode. {\color{black}In our experiments, we employ Gaussian random sketching matrices and set the oversampling parameter to $p=2$.} As for the HOSVD, we implement it in \texttt{Python} using the \texttt{NumPy} library for the classical matrix operations and the \texttt{TensorLy} library to perform the tensor matricization and the TTM product.
    \item [Randomized HOSVD with Kronecker product Reusing factor (rHKron)]~\cite{Minster2024ReKron}, that is a randomized algorithm consisting in two main steps. Firstly, for each mode $k$, the input tensor is sketched along all modes except the $k$-th one, and an orthogonal basis approximating  
    the range of its matricization along mode $k$ is computed by QR. Once all the $d$ Q-factors are computed, they are used to reduce the input tensor via a TTM along all modes. The resulting reduced tensor is decomposed by a\textcolor{black}{n} ST-HOSVD. The ST-HOSVD core is returned as core tensor of this decomposition, while the ST-HOSVD factor matrices are multiplied by the Q-factors to form the Tucker factors. We implement this algorithm in \texttt{Python} using the \texttt{NumPy} library for the classical matrix operations and the \texttt{TensorLy} library to perform the tensor matricization and the TTM product.
\end{description}
All these methods are compared in terms of approximation accuracy and computational efficiency. Results for the randomized algorithms are reported through boxplots and logarithmic-scale plots of the approximation error and computational time, computed over $25$ independent runs, in order to assess both typical performance and variability.
{\color{black} When we report the results of Sub-R-HOSVD, we also include a scaling factor that determines the number of selected fibers along each mode. In particular, we use labels of the form ``$\alpha$ Sub-R-HOSVD''. This means that, for mode $k$, the algorithm selects a number of fibers proportional to the mode size $n_k$, with the reported value $\alpha$ acting as the proportionality coefficient, namely $s_k=\alpha n_k$. This choice ensures a consistent sampling rate across modes while adapting to the dimensionality of the data, without explicitly forming tensor unfoldings.}

All experiments, both serial and parallel, were performed on a single node of the Booster partition of Leonardo~\cite{leonardo}, an HPC system hosted at Cineca\footnote{More details can be found at \url{https://www.hpc.cineca.it/systems/hardware/leonardo/}}. More in detail, we used an Atos BullSequana X2135, Da Vinci single-node GPU, equipped with one Intel Ice Lake Xeon Platinum 8358 (32-core) CPU processor and 512\,GiB DDR4 3200\,MHz RAM. The node features four NVIDIA Ampere A100 GPUs (custom), 64\,GiB HBM2e with NVLink 3.0 (200\,GB/s) accelerators. 
%
%

\subsubsection{Synthetic tests}\label{sec:tucker:exp:synthetic}
In the synthetic experiments, we consider low-rank tensors of mode dimensions and rank components all equal to $15$ and $5$, respectively, and varying order, from $d=4$ to $d=8$. {The core tensor entries come from the evaluation of the function in \cite[Equation $(37)$]{Ahmadi2021ReviewRandTucker} on a multidimensional grid, while the Tucker factors come from the orthogonalization of randomly generated matrices with i.i.d. entries from the normal distribution. After generating the core and the factors, dense tensors are reconstructed and then decomposed by the considered algorithms, to test the ability of the latter ones in computing accurate low-rank decompositions. 

For each tensor order, boxplots summarize the distribution of errors and computational times, while logarithmic-scale plots show the median performance as a function of $d$, highlighting the good scaling behavior of the methods. As mentioned above, for these problems we run Algorithm~\ref{alg:Sub-R-HOSVD} in a sequential manner, namely the loops are performed one after the other. Indeed, one of the goals of this section is to show how our computational improvements lead to dramatic speed\textcolor{black}{-}ups of the tensor compression even when not working with proper parallel computing architectures.

The results are consistent across all the different tests and can be summarized as follows. In terms of running time, our Sub-R-HOSVD is always very competitive to the point that it becomes one order of magnitude faster than most of the other routines for a large number of modes $d\geq 5$. Only when compared to ST-HOSVD, this performance gap becomes narrower. Indeed, the sequential reductions of the tensor performed by ST-HOSVD allow for the computation of the SVDs of objects of decreasing size, with a consequent cost reduction. Nevertheless, Sub-R-HOSVD is still faster than ST-HOSVD, in addition to \textcolor{black}{being} extremely suitable for parallelization as we will show in section~\ref{sec:experiments:parallel}.

The good performance of Sub-R-HOSVD is mainly due to our subsampling step that avoids handling any large matricization. Moreover, it is interesting to see how an extremely small number of fibers is indeed enough to get very accurate results for these examples, where the factorized tensors are low-rank in the Tucker sense by construction. In particular, when $d=8$ and the subsampling parameter is set to 75 {\color{black}(cf. 5 Sub-R-HOSVD)}, it means that, for each mode, we are employing only 75 columns of the current matricization instead of using all its $n^7=170\,859\,375$ columns. \textcolor{black}{Sub-R-HOSVD} thus works with 0.00005\% of the data without \textcolor{black}{incurring} 
any loss of accuracy. We also highlight that it looks quite the opposite. In particular, the error achieved by \textcolor{black}{Sub-R-HOSVD} remains \textcolor{black}{approximately} constant when increasing the number of modes $d$. This does not happen for HOSVD, ST-HOSVD, \textcolor{black}{and} R-HOSVD,
whose error \textcolor{black}{poorly} scales with $d$. We conjecture that this different trend is due to handling much smaller objects in our setting. This feature, \textcolor{black}{which} is \textcolor{black}{also shared} by the rHKron scheme, may alleviate the accumulation of numerical error in finite precision arithmetic. We believe this interesting phenomenon deserves further study that we plan to carry out in the near future.

\begin{figure}[t!]
\centering
\subfloat[Median relative error \textcolor{black}{on the y-axis}.]{\label{fig:Tuck:37:medianerr}
    \includegraphics[width=0.4\textwidth]
    {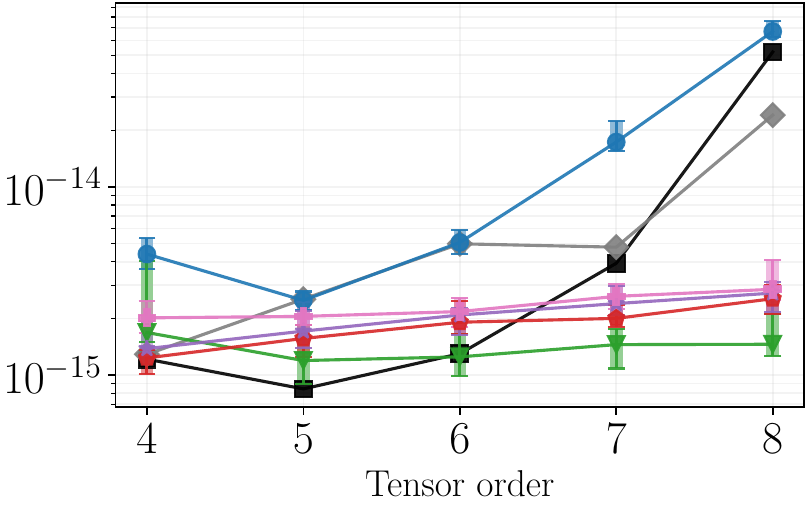}}
\hspace{15pt}
\subfloat[Median computational time \textcolor{black}{on the y-axis}.]{\label{fig:Tuck:37:mediantime}
    \includegraphics[width=0.4\textwidth]
    {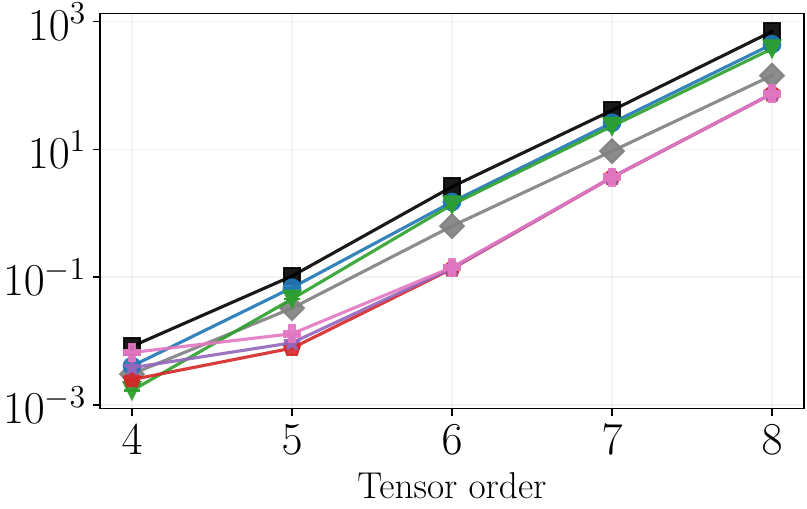}}
\vspace{5pt}
\subfloat[Boxplot of the relative errors for each order.]{\label{fig:Tuck:37:medianerr:BOX}
    \includegraphics[width=0.99\linewidth]{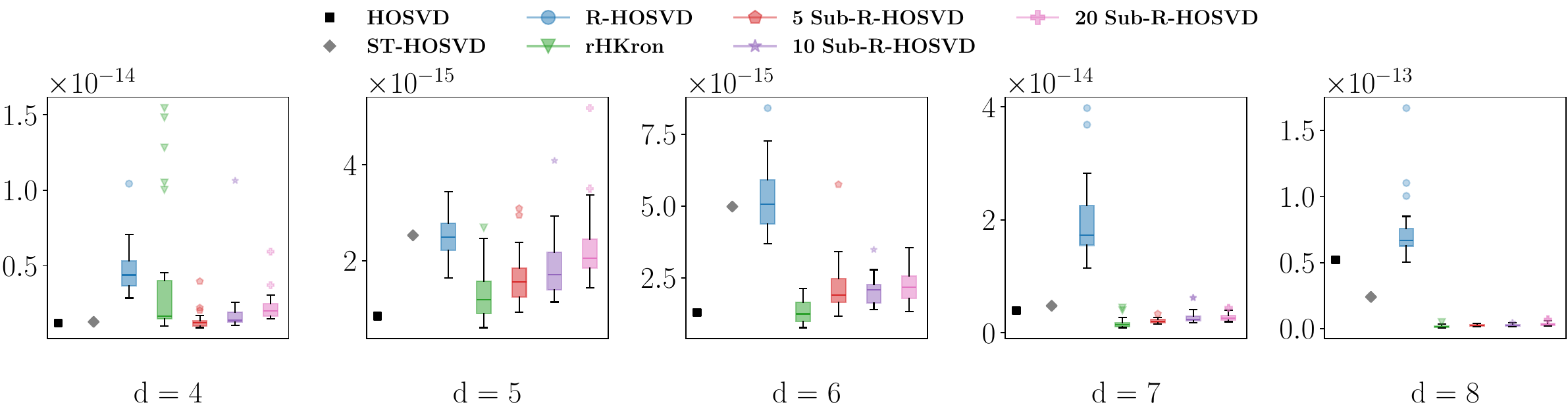}}
\caption{Comparison of the considered methods on synthetic tensors with core entries generated as in~\cite[Equation $(37)$]{Ahmadi2021ReviewRandTucker}. All the randomized algorithms are tested over $25$ independent runs.}
\label{fig:Tuck:37}
\end{figure}

{\color{black}
We conclude this example by documenting the memory consumption of our Sub-R-HOSVD and comparing it to that of the deterministic routines HOSVD and ST-HOSVD for $d\geq6$; for smaller $d$ all these methods employ similar storage allocation. The results are reported in Table~\ref{table:memory}. For a given method, the column \emph{Baseline} collects the overall storage 
measured prior the call to that routine, \emph{Peak} represents the peak in the memory requirements needed by the corresponding computational strategy, and \emph{Extra} is just the difference between \emph{Peak} and 
\emph{Baseline}, namely the actual memory a given routine needs to complete its task. All the reported storage demands are in Gigabytes.
These results show how Sub-R-HOSVD requires only a fraction of the memory needed by the deterministic methods HOSVD and ST-HOSVD. Indeed, our randomization-based scheme asks for about 66\% less memory than HOSVD and ST-HOSVD. 

\begin{table}[t!]
\centering
{\color{black}
{
\begin{tabular}{r| rrr|r rr| rrr}
    &  \multicolumn{3}{c|}{HOSVD}& \multicolumn{3}{c|}{ST-HOSVD} & \multicolumn{3}{c}{Sub-R-HOSVD} \\
$d$ & Baseline & Peak & Extra & Baseline & Peak & Extra & Baseline & Peak & Extra \\
6 & 0.31 & 0.68& 0.36 & 0.33& 0.65 & 0.32& 0.35& 0.45 & 0.10\\
7 & 1.61 & 6.73 & 5.12 & 1.61&  6.82 & 5.21 &1.65& 3.37 & 1.72\\
8 & 19.86 & 98.07 & 78.21 & 19.87& 97.93 &78.05 &19.91&  45.99 & 26.07\\
\end{tabular}
\caption{{\color{black}Memory demand required by HOSVD, ST-HOSVD, and Sub-R-HOSVD with a fixed subsampling parameter $s=20\cdot n$. Baseline: storage prior the call of the corresponding routine. Peak: maximum memory usage. Extra: Peak minus Baseline. All the reported values are in Gigabytes (GB).}}\label{table:memory}
}}
\end{table}

}
\subsubsection{Real-world examples}\label{sec:exp:realworld}
{\color{black}We now repeat similar experiments as above but on a tensor storing real-world data. The main point of this test is to show that the sampling step does not lead to any serious loss in the accuracy of the computed approximation, also when dealing with application-oriented datasets. This suggests that our Sub-R-HOSVD can be a strong ally to reduce the computational cost and memory requirements of tensor compression in real-world applications where the dimensionality of the datasets at hand may make any operation involving \textcolor{black}{their} full representation extremely costly in practice.

We use a meteorological dataset of Northern Italy obtained from the Copernicus Emergency Management Service (CEMS) Early Warning Data Store\footnote{\url{https://ewds.climate.copernicus.eu/}}, and arrange the data into a tensor of shape $(5,10,12,112,80,270)$. {\color{black}These six dimensions correspond to five environmental variables (River discharge in the last 6 hours, Runoff water equivalent (surface plus subsurface), Snow depth water equivalent, Volumetric soil moisture, and Soil wetness index (root zone)), ten years of observations (2014--2023), twelve months, 112 six-hourly observations per month (obtained by truncating each month to 28 days, yielding $28\times4=112$ observations), and an $80\times270$ spatial grid covering the physical domain. In this case, we truncate the input tensor to multilinear rank $(5,8,10,28,50,150)$, corresponding to approximately $5.8\%$ of the size of the original tensor}, using the T-HOSVD rather than the HOSVD. For the randomized algorithms, results are reported in Figure~\ref{fig:Tuck:meteo} using boxplots of error and computational time over 25 independent runs.

As for the results in section~\ref{sec:tucker:exp:synthetic}, also in this case, Sub-R-HOSVD is one order of magnitude faster than most of the other methods while achieving very similar accuracy records. As before, the performance of ST-HOSVD turns out to be closer to that of Sub-R-HOSVD with the latter algorithm being consistently faster than the former. }

\begin{figure}[t!]
\centering
\subfloat[Boxplot of the relative errors.]{\label{fig:Tuck:meteo:BOXerr}
\includegraphics[height=0.199\linewidth]{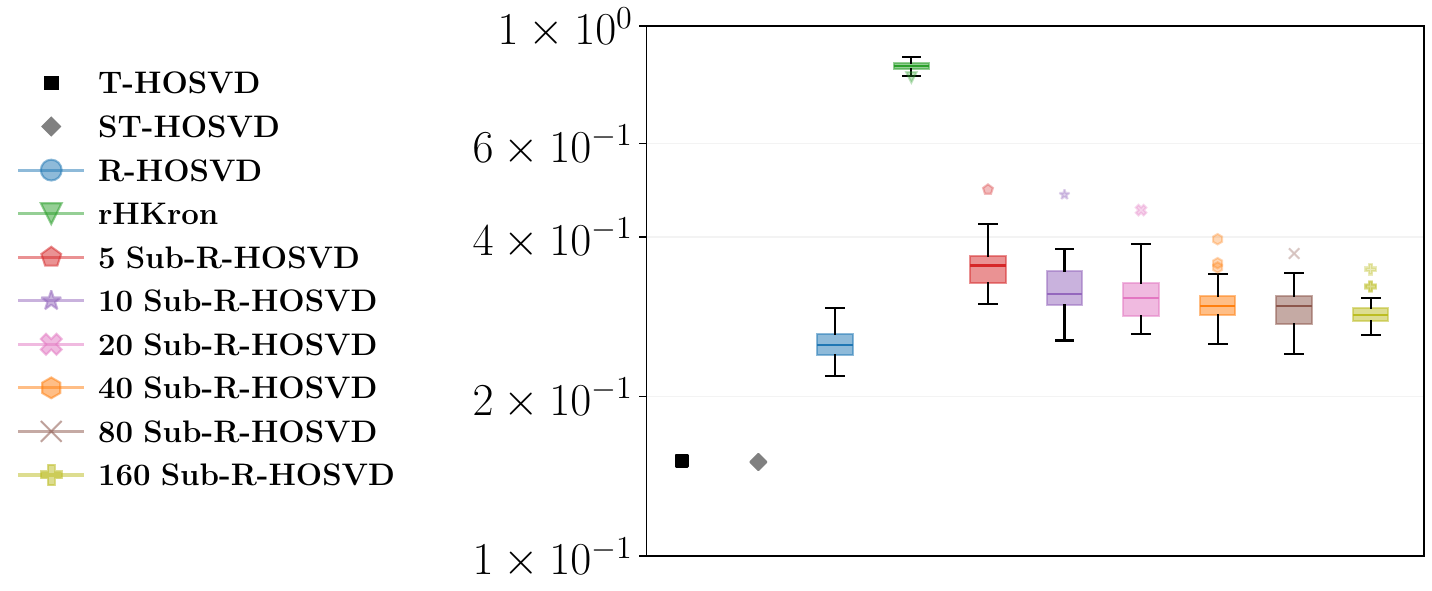}}
\subfloat[Boxplot of the computational times.]{\label{fig:Tuck:meteo:BOXtime}
\hspace{40pt}
\includegraphics[height=0.192\linewidth]{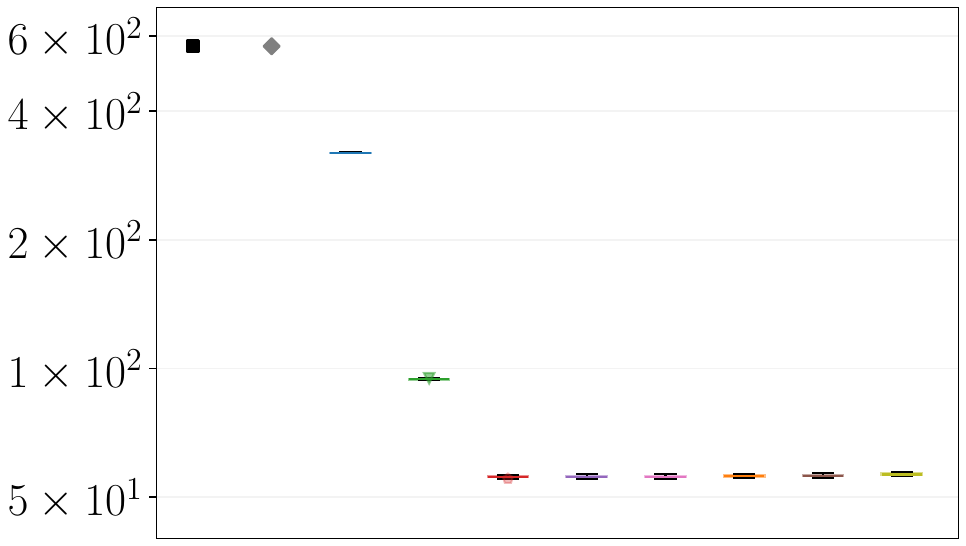}}
\caption{Comparison of the considered methods on the tensor constructed constructed from the meteorological data over Northern Italy obtained from the CEMS Early Warning Data Store. All the randomized algorithms are tested over $25$ independent runs.}
\label{fig:Tuck:meteo}
\end{figure}

\subsubsection{Parallel experiment}\label{sec:experiments:parallel}

The goal of this section is to assess the parallel scalability of Sub-R-HOSVD. 
The parallel implementation of our scheme consists of two main stages. 
The first stage computes the orthogonal factor matrices {\color{black}$\mat{U}_k$}, $k=1,\dots,d$, in Algorithm~\ref{alg:Sub-R-HOSVD} in parallel.
The second stage performs the successive TTM operations 
(line~\ref{alg:subRHOSVD_computecore} in Algorithm~\ref{alg:Sub-R-HOSVD}) to 
progressively reduce the tensor and compute the core tensor.
To this end, the input tensor is first partitioned into a number of slices equal to the number 
of available processes along a suitable mode $\bar{k}$, chosen to ensure a balanced 
workload. Each slice is assigned to a different process. 
Each process then performs, in serial, all TTM operations along the modes except 
for mode $\bar{k}$, producing a locally reduced tensor $\bar{\ten{S}}$. The master 
process subsequently gathers all reduced tensors and performs the final TTM along 
mode $\bar{k}$ to obtain the core tensor $\ten{S}$. Consequently, the only serial 
component of this stage is the last TTM executed after the collection of all $\bar{\ten{S}}$. 

We test the parallel implementation of Sub-R-HOSVD on a synthetic low-rank tensor, constructed as the one of section~\ref{sec:tucker:exp:synthetic}, with order $d=8$ and rank $5$ along each mode, and mode dimension varying from $n=16$ to $n=20$. 
For each mode $k=1,\ldots,d$, the subsampling parameter is set to $s\equiv s_k=10^4$. This means that the percentage of columns of the matricization $\mat{X}^{(k)}$ we employ varies from $0.004\%$ (for $n=16$) to $0.0008\%$ (for $n=20$) for each $k=1,\ldots,d$. 

In Figure~\ref{fig:Tuck:PAR:time} and Figure~\ref{fig:Tuck:PAR:speedup} we report strong scaling and speed-up results, respectively, as the number of processes doubles, from $1$ up to $8$ while the problem size $n$ is fixed. 
Notice the logarithmic scale in~Figure~\ref{fig:Tuck:PAR:time}. In particular, the time devoted to the serial component of the code, i.e., the last TTM mentioned above, is at least one order of magnitude smaller than the total execution time, making the former negligible in the overall parallel performance of the method. In Figure~\ref{fig:Tuck:PAR:time}, for each $n$ and number of adopted processes we report the total execution time in seconds (rounded to the nearest integer) on top of the related {\color{black}bars along with the running time of each main component inside the corresponding bar}. This helps us \textcolor{black}{appreciate} the very good strong scaling of our numerical scheme. On the other hand, by looking at Figure~\ref{fig:Tuck:PAR:speedup}, we can notice that the overall speed\textcolor{black}{-}up worsens as $n$ increases. A thorough inspection of the performance of our code shows that this is due to the construction of the \textcolor{black}{indices} for the fiber sampling. Indeed, this operation does not scale linearly with $n$. To see this, let \textcolor{black}{us} consider $n=19$. By recalling that $d=8$ (and fixed), we need to draw $s$ integers in $[1,n^7]$ with $n^7=19^7\approx8.93\cdot10^8$. If we increase $n$ to $n=20$, we still pick $s$ integers but
now in a much wider interval as $20^7\approx1.28\cdot10^9$ and this operation is more expensive than before, when $n$ was $19$. 

To fix this issue, we \textcolor{black}{also parallelize} the generation of the integers for the fiber sampling, as follows. For a given $n$, we divide the interval $[1,n^{d-1}]$ into four subintervals of similar length: $[1,\lfloor n^{d-1}/4\rfloor]$, $[\lfloor n^{d-1}/4\rfloor]+1,2\lfloor n^{d-1}/4\rfloor]$, $2[\lfloor n^{d-1}/4\rfloor]+1,3\lfloor n^{d-1}/4\rfloor]$, and $3[\lfloor n^{d-1}/4\rfloor]+1,4\lfloor n^{d-1}/4\rfloor+\text{mod}(n^{d-1},4)]$. From the first three intervals we randomly pick $\lfloor s/4\rfloor$ integers whereas
we pick $\lfloor s/4\rfloor+\text{mod}(s,4)$ of them from the last interval. We run the generation of these \textcolor{black}{four} sets of integers, and the extraction of the related fibers, in parallel. We report the results obtained by adopting this strategy in Figure~\ref{fig:Tuck:PAR_parextraction}. In particular, looking at Figure~\ref{fig:Tuck:PAR:time_bis} we can see that the overall running times are lower than the ones reported in Figure~\ref{fig:Tuck:PAR:time} for every $n$ and number of employed processes. More remarkably, Figure~\ref{fig:Tuck:PAR:time_bis} showcases the much better speed\textcolor{black}{-}up we obtain by running also the \textcolor{black}{indices} generation in parallel. In particular, the speed up related to the computation of the factors {\color{black}$\mat{U}_k$} (light blue line with circles) gets very close to the optimal speed up (dashed black line), also for $n=20$. This is clearly beneficial also in terms of the overall speed up (yellow line with triangles).

We conclude this section with a couple of comments. First, the reasoning behind splitting $[1,n^{d-1}]$ into four parts is due to the computational resources we have available. Indeed, if we split this interval into $\ell$ subintervals, and $p$ denotes the number of processes we have available for computing the {\color{black}$\mat{U}_k$} factors, then the overall number of processes we actually need for the parallel generation of the indices is $\ell p$. We have a maximum of 32 processes at our disposal which means we can divide $[1,n^{d-1}]$ \textcolor{black}{into} $\lfloor 32/p\rfloor$ subintervals, at most. For the maximum number of processes we employ for the computation of the {\color{black}$\mat{U}_k$} factors, i.e., $p=8$\footnote{We recall that employing a number of processes larger than $d$ would not make much sense for the computation of the {\color{black}$\mat{U}_k$} factors.}, this means that we can afford to split $[1,n^{d-1}]$ \textcolor{black}{into} at most $\ell=32/8=4$ subintervals.

The second comment is about the selection of the integers for the fiber sampling. It is true that randomly picking $s$ integers in $[1,n^{d-1}]$ is not equivalent to what we do with our parallel computation. This means that the assumptions we had in section~\ref{sec:sr-HOSVD} for deriving our error bounds are not exactly satisfied and some technical adjustments are probably needed. However, since in section~\ref{sec:sr-HOSVD} we assumed a uniform \textcolor{black}{draw} from $[1,n^{d-1}]$, the strategy of dividing the latter interval into subintervals and then uniformly picking integers from those may be seen as a good approximation of the original integer generation and no significant difference has been observed in practice.
{\color{black} We would like also to mention that different strategies for achieving an efficient fiber index generation may be adopted as well; see, e.g., the supplementary material for further results coming by employing one of those.}

\ifbool{showfigures}{
\begin{figure}[t!]
\centering
\subfloat[Strong scaling as a function of the number of processes.]{\label{fig:Tuck:PAR:time}\hspace{-10pt}
\includegraphics[width=1.02\linewidth]{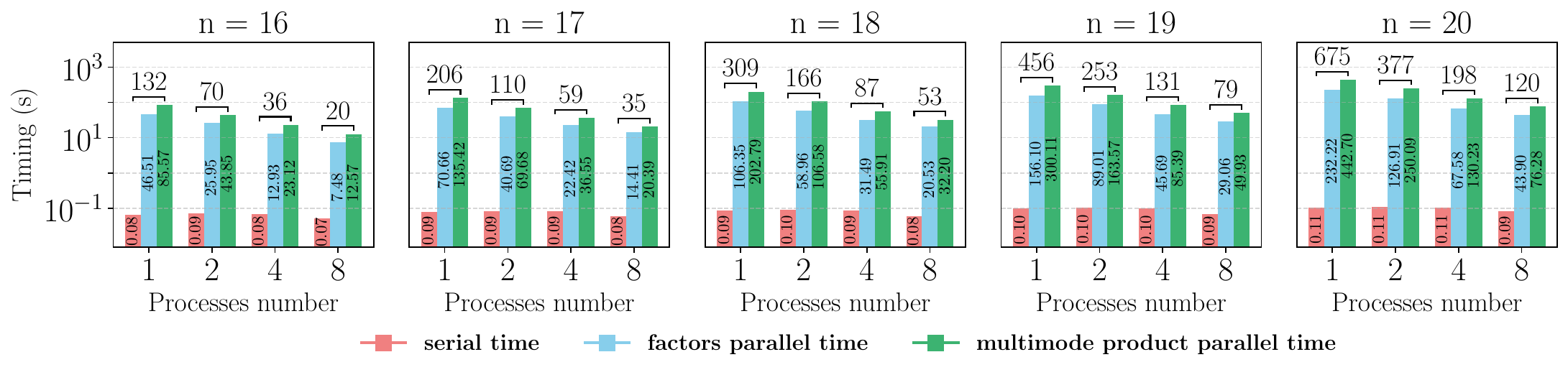}} \\
\subfloat[Speed-up  as a function of the number of processes.]{\label{fig:Tuck:PAR:speedup}
\includegraphics[width=0.99\linewidth]{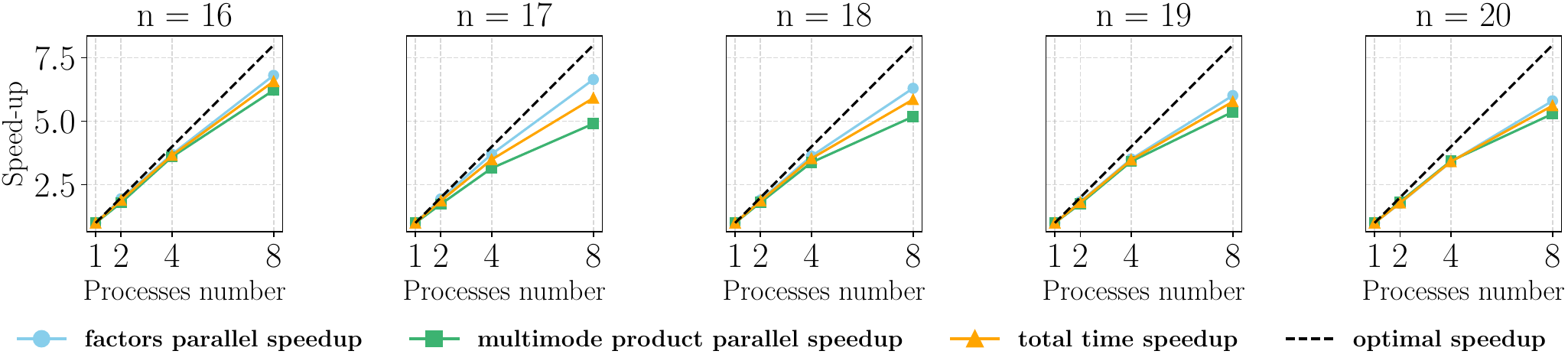}}
\caption{{Parallel performance of Sub-R-HOSVD. The dashed line with slope~1 denotes the ideal linear speed-up.}}
\label{fig:Tuck:PAR}
\end{figure}}{
\centering{\color{black} figures off}
}

\ifbool{showfigures}{
\begin{figure}[t!]
\centering
\subfloat[Strong scaling as a function of the number of processes.]{\label{fig:Tuck:PAR:time_bis}\hspace{-10pt}
\includegraphics[width=1.02\linewidth]{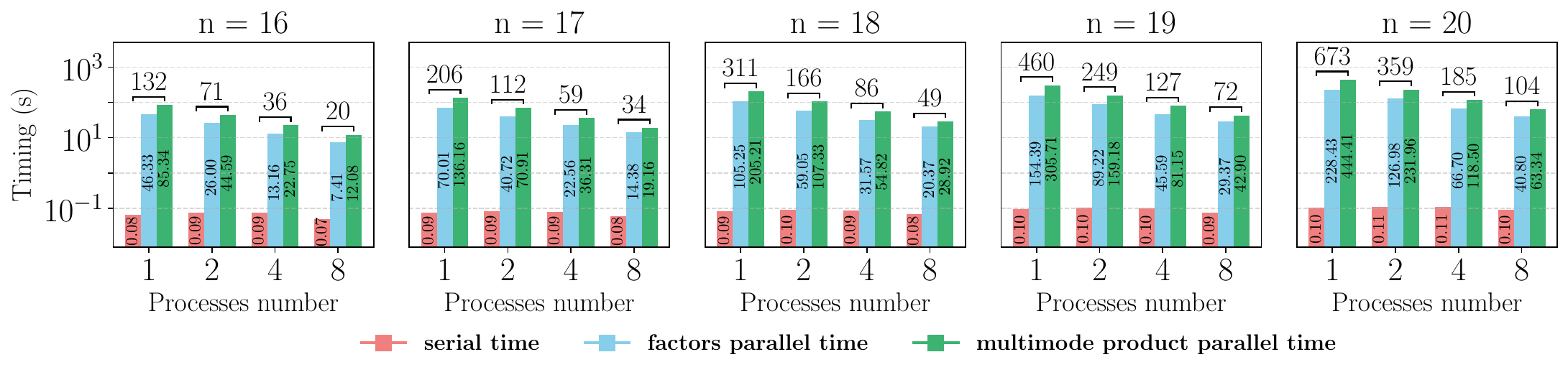}} \\
\subfloat[Speed-up  as a function of the number of processes.]{\label{fig:Tuck:PAR:speedup_bis}
\includegraphics[width=0.99\linewidth]{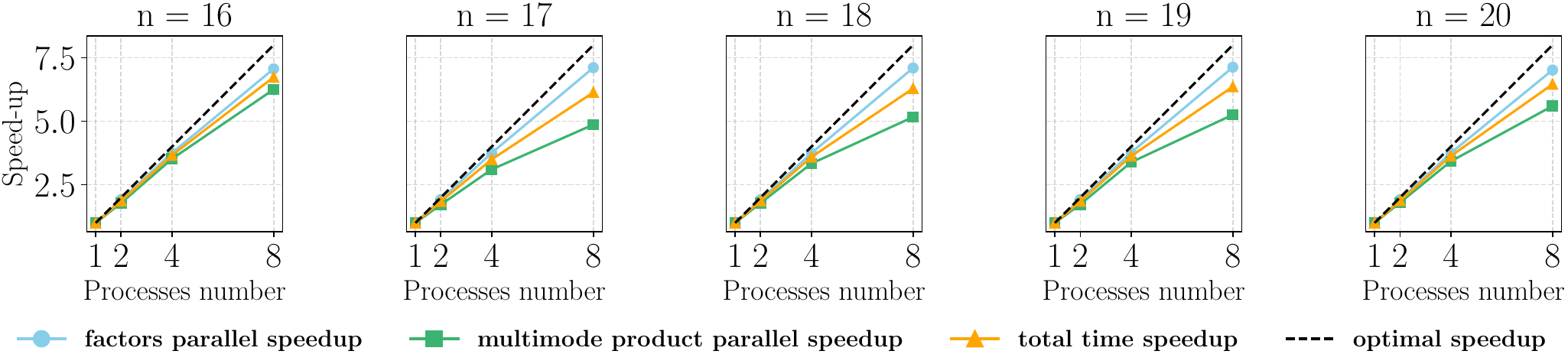}}
\caption{{Parallel performance of Sub-R-HOSVD with a parallel computation of the indeces for the fiber sampling. The dashed line with slope~1 denotes the ideal linear speed-up.}}
\label{fig:Tuck:PAR_parextraction}
\end{figure}}{
\centering{\color{black} figures off}
}



\section{The H-Tucker decomposition}\label{H-Tucker}

\begin{figure}
\centering
\scalebox{0.8}{
\begin{tikzpicture}[
    scale = 0.8,
    level distance=2cm,
    level 1/.style={sibling distance=10cm},
    level 2/.style={sibling distance=5cm},
    level 3/.style={sibling distance=2.5cm},
    set node/.style={ellipse, draw, thick, minimum height=0.8cm, inner sep=3pt},
    leaf node/.style={circle, draw, thick, minimum size=0.8cm, inner sep=0pt},
    tag node/.style={draw=none, rectangle, font=\large}
]

\node[set node] (root) {$\large 1,2,3,4,5,6,7,8$}
    child {
        node[set node] (n1) {$1,2, 3,4$}
        child {
            node[set node] (n11) {$1,2$}
            child { node[leaf node] (n111) {$1$} }
            child { node[leaf node] (n112) {$2$} }
        }
        child {
            node[set node] (n12) {$3,4$}
            child { node[leaf node] (n121) {$3$} }
            child { node[leaf node] (n122) {$4$} }
        }
    }
    child {
        node[set node] (n2) {$5,6, 7,8$}
        child {
            node[set node] (n21) {$5,6$}
            child { node[leaf node] (n211) {$5$} }
            child { node[leaf node] (n212) {$6$} }
        }
        child {
            node[set node] (n22) {$7,8$}
            child { node[leaf node] (n221) {$7$} }
            child { node[leaf node] (n222) {$8$} }
        }
    };

\node[tag node] at ([yshift=-1cm]root.center) {$\ten{B}_0$};
\node[tag node] at ([yshift=-1cm]n1.center) {$\ten{B}_{1}$};
\node[tag node] at ([yshift=-1cm]n2.center) {$\ten{B}_2$};
\node[tag node] at ([yshift=-1cm]n11.center) {$\ten{B}_3$};
\node[tag node] at ([yshift=-1cm]n12.center) {$\ten{B}_4$};
\node[tag node] at ([yshift=-1cm]n21.center) {$\ten{B}_5$};
\node[tag node] at ([yshift=-1cm]n22.center) {$\ten{B}_6$};
\end{tikzpicture}
}
\caption{Dimension tree for an order-$8$ tensor with transfer tensors enumerated by heap indexing.}
\label{fig:HT:dimtree}
\end{figure}
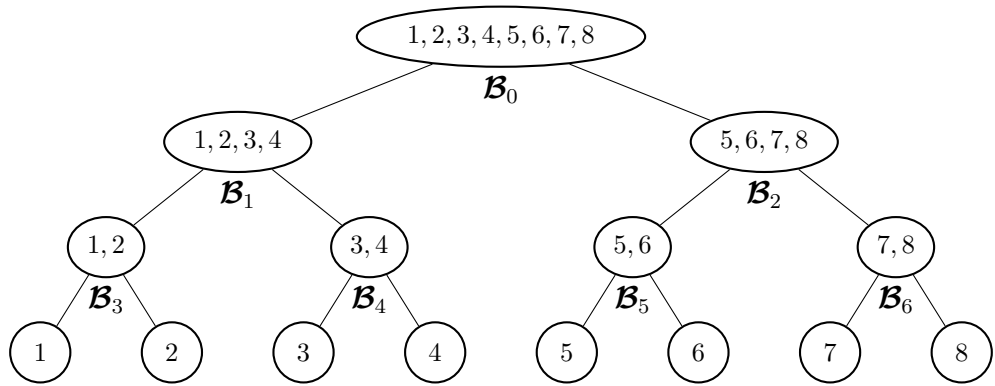

The Hierarchical Tucker (H-Tucker) decomposition relies on a \emph{dimension tree} to decompose order-$d$ tensors.
\begin{definition}
    A dimension tree $\set{T}$, associated with an order-$d$ tensor $\ten{X}\in\R^{\vec{n}}$, is a binary tree of depth $z = \lceil{\log_2 d}\rceil$ such that: 
    \begin{enumerate}
        \item the root is a set including all the tensor mode indices, $\set{R} = \{1,\dots, d\}$; 
        \item all leaves are singleton\textcolor{black}{s}, $\set{L}_k = \{k\}$ for $k=1,\dots, d$;
        \item each internal node $\set{S}$ has two disjoint children, that is \textcolor{black}{$\set{S} = \set{S}_{\ell}\, {\cup}\,\set{S}_{r}$ and\\ $\set{S}_{\ell}\cap  \set{S}_{r} = \emptyset$}.
    \end{enumerate}
    Let \textcolor{black}{us} assume the nodes of $\set{T}$ to be enumerated. Then, the \emph{hierarchical rank} of $\ten{X}$ with respect to the dimension tree $\set{T}$ is {\color{black}$\vec{r}=(r_1,\ldots,r_z)^T\in\N^{z}$} such that {\color{black}$r_k = \text{rank}(\mat{X}^{(\set{\footnotesize S})})$} where $\set{S}\in\set{T}$ is the $k$-th node.
\end{definition}
An example of \textcolor{black}{a} dimension tree is depicted in Figure~\ref{fig:HT:dimtree} for a generic order-$8$ tensor. To simplify the exposition, we denote by ${r}_\set{S}$ and $r_k$ the hierarchical rank of $\ten{X}$ associated with the internal node $\set{S}$ and the leaf $\{k\}$, respectively. Additionally, we define the \textcolor{black}{leaf} rank as $\vec{r}_{\set{L}}:= (r_1, \dots, r_d)$.

Let $\ten{X}\in\R^{\vec{n}}$ be an order-$d$ tensor, $\set{T}$ a dimension tree of depth $z$ and $\vec{r}\in\R^{z}$ the hierarchical rank of $\ten{X}$ associated with the dimension tree $\set{T}$\textcolor{black}{;} the H-Tucker associates an orthogonal matrix $\mat{U}_k$ to each leaf $\set{L}_{k}$ such that ${\rm Range}(\mat{U}_k) = {\rm Range}(\mat{X}^{(k)})$. Similarly, it associates an orthogonal matrix $\mat{U}_{\set{\footnotesize S}}$ of size $(n_{\set{\footnotesize S}} \times r_{\set{\footnotesize S}})$ with each internal node $\set{S}$ {\color{black}such that {\rm (i)} ${\rm Range}(\mat{U}_{\set{\footnotesize S}}) = {\rm Range}(\mat{X}^{(\set{\footnotesize S}\;)})$ and {\rm (ii)} ${\rm Range}(\mat{U}_{\set{\footnotesize S}}) \subseteq {\rm Range}(\mat{U}_{\set{\footnotesize S}_\ell} \kron \mat{U}_{\set{\footnotesize S}_r})$,}
where $\kron$ denotes the Kronecker product, and $\mat{U}_{\set{\footnotesize S}_\ell}$ and  $\mat{U}_{\set{\footnotesize S}_r}$ are $(n_{\set{\footnotesize S}_\ell} \times r_{\set{\footnotesize S}_\ell})$ and $(n_{\set{\footnotesize S}_r} \times r_{\set{\footnotesize S}_r})$ orthogonal matrices associated with the children of $\set{S}$, i.e., $\set{S}_\ell$ and $\set{S}_r$. 
The second condition describes the \emph{nestedness property}, namely that the subspace spanned by the parent node basis is included in the subspace spanned by the Kronecker product of its children bases. This property guarantees that we can express the columns of $\mat{U}_{\set{\footnotesize S}}$ as a linear combination of the columns of the \textcolor{black}{leaf} orthogonal bases, reducing the storage costs. Recall that the \textcolor{black}{leaf} orthogonal bases are $n_k \times r_k$ matrices, while the internal \textcolor{black}{node} bases have $n_{\set{\footnotesize S}}$ rows. Thus, if by using the nestedness property we find a way to avoid storing the internal \textcolor{black}{node} bases, then we can reduce the storage complexity. The solution proposed in~\cite{Grasedyck2009} is storing the matrix that allows us to express each internal node basis as a linear combination of the Kronecker product of its children bases, that is\textcolor{black}{,} the matrix $\mat{B}_{\set{\footnotesize S}}$ is such that
$\mat{U}_{\set{\footnotesize S}}= (\mat{U}_{\set{\footnotesize S}_\ell} \kron \mat{U}_{\set{\footnotesize S}_r})\mat{B}_{\set{\footnotesize S}}$,
for every internal node $\set{S}$. In~\cite{Grasedyck2009}, the entries of $\mat{B}_{\set{\footnotesize S}}$ are organized into 
%
 an order-$3$ tensor, $\ten{B}_{\set{\footnotesize S}}$, of size $(r_{\set{\footnotesize S}}, r_{\set{\footnotesize S}_{\ell}}, r_{\set{\footnotesize S}_r})$, called \emph{transfer tensor}, such that 
$\mat{U}_{\set{\footnotesize S}} = \bigl(\mat{U}_{\set{\footnotesize S}_\ell}\kron \mat{U}_{\set{\footnotesize S}_r}\bigr)\bigl(\mat{B}_{\set{\footnotesize S}}^{\;(1)}\bigr)^{T}$, 
where $\mat{B}_{\set{\footnotesize S}}^{\;(1)}$ denotes the mode-1 unfolding of $\ten{B}_{\set{\footnotesize S}}$.
Thanks to this construction, the H-Tucker factorizes an order-$d$ tensor into $d-1$ transfer tensors of order-$3$, and $d$ \textcolor{black}{leaf} factor matrices, reducing the overall storage costs from $\bigO(n^d)$ to $\bigO(drn + (d-1)r^2)$ where $n$ is the largest dimension of $\ten{X}$ and $r$ is the largest rank of $\mat{X}^{(\set{\footnotesize S})}$ for every node of the dimension tree.

In~\cite{Grasedyck2009}, two algorithms are proposed to compute the H-Tucker decomposition: the root-to-leaves (RtL-HT) and the leaves-to-root (LtR-HT). The RtL-HT has an algorithmic scheme affine to the T-HOSVD, that is, all the orthogonal bases are computed from the input dense tensor; the LtR-HT has an algorithmic scheme affine to the ST-HOSVD, that is, every time an orthogonal basis is formed, it is used to reduce the size of the input tensor. To be consistent with the framework presented in the first half of this manuscript, we focus on the RtL-HT algorithm. 
Given an input order-$d$ tensor $\ten{X}\in\R^{\vec{n}}$ and a dimension tree $\set{T}$ of depth $z$, the \textcolor{black}{leaf} factors are the Tucker factor matrices computed via the HOSVD of $\ten{X}$ at multilinear rank $\vec{r}_{\set{L}}$. The core tensor is not formed. 
Then, the algorithm visits all the internal nodes moving toward the root, from level $z-1$ to level $1$. For every internal node $\set{S}$ we compute the SVD of $\mat{X}^{(\set{\footnotesize S})}$, forming $\mat{U}_\set{\footnotesize S}$ with the $r_{\set{\tiny S}}$ dominant left-singular vectors. The transfer tensor between node $\set{S}$ and its children $\set{S}_\ell$ and $\set{S}_r$ is computed as 
\begin{equation}
\label{eq:transften:comp}
    \ten{B}_{\set{\footnotesize S}}(i,j,k) = \langle \vec{u}_i, \vec{v}_j\kron \vec{w}_k\rangle,
\end{equation}
where $\vec{u}_i$ is the $i$-th column of $\mat{U}_{\set{\footnotesize S}}$, $\vec{v}_j$ is the $j$-th column of $\mat{U}_{\set{\footnotesize S}_\ell}$, and $\vec{w}_k$ is the $k$-th column of $\mat{U}_{\set{\footnotesize S}_r}$. 
Once all the internal nodes at level $z-1$ have been visited, the algorithm repeats these steps for the internal nodes at level $z-2$, until reaching the root of the dimension tree. Once 
level $0$ is reached, the transfer tensor between the root and its two children is computed as
\begin{equation}
\label{eq:transftenR:comp}
    \ten{B}_{\set{\footnotesize R}}(1,j,k) = \langle \text{vec}(\ten{X}), \vec{v}_j\kron \vec{w}_k\rangle,    
\end{equation}
where $\vec{v}_j$ is the $j$-th column of $\mat{U}_{\set{\footnotesize R}_\ell}$, and $\vec{w}_k$ is the $k$-th column of $\mat{U}_{\set{\footnotesize R}_r}$.  

Assuming to compute the RtL-HT of an order-$d$ tensor of size $\vec{n}$ along each mode and using~\cite[Lemma 3.19]{Grasedyck2009}, we conclude that 
the computational complexity of the RtL-HT algorithm corresponds to the cost of $2(d-1)$ SVDs, that is
 $   \mathcal{O}\Bigl(\bigl(\prod_{i=1}^dn_i\bigr)^{3/2}\Bigr)$.
This computational complexity reduces to $\mathcal{O}(n^{3d/2})$ if the modes are all equal in size. 

In~\cite[Remark 3.12]{Grasedyck2009}, the author proves that given an order-$d$ tensor $\ten{X}\in\R^{\vec{n}}$ and a hierarchical rank $\vec{r}$ associated with a complete binary dimension tree $\set{T}$, if $\widehat{\ten{X}}$ denotes the approximation of $\ten{X}$ produced by the RtL-HT algorithm using $\set{T}$ and $\vec{r}$, then
\[
\|{\ten{X} - \widehat{\ten{X}}}\|_F \le \sqrt{2d-3} \norm{{\ten{X}}-{\ten{X}}^{\star}}_F, 
\]
where $\ten{X}^{\star}$ denotes the best approximation of $\ten{X}$ at the hierarchical rank $\vec{r}$.

 To the best of our knowledge, no contribution available in the current literature has ever proposed the use of RSVD within either the RtL-HT or the LtR-HT algorithm. While its implementation is trivial -- just replace each SVD within {\color{black}RtL-HT/LtR-HT} with \textcolor{black}{an} RSVD -- we believe its analysis is far from easy. Indeed, in the H-Tucker, the quality of the matrix $\mat{U}_{\set{S}}$ necessarily impacts the quantities constructed in the following levels of the tree, {\color{black}either directly as in LtR-HT or implicitly via the transfer tensors as in RtL-HT}. Tracking down the inexactness introduced by RSVD at a certain level of the tree, studying its impact on the objects computed at the following levels, and assessing the overall accuracy of the scheme would be a significant scientific contribution itself\textcolor{black}{,} and we refrain from carrying it out here. Therefore, we do not have any error bound achieved by the subsampled randomized  H-Tucker scheme we propose in the next section, as we did for Sub-R-HOSVD; cf.~Theorem~\ref{theorem:errorbound_SubRHOSVD}. The main goal of the following sections is to design this scheme and show its computational performance. We leave its challenging, yet important, error analysis \textcolor{black}{for future work}.

\subsection{Subsampled randomized RtL-HT}\label{Subsampled randomized LtR-HT}
In this section, we present the second most important contribution of this paper. Our goal is to further reduce the complexity of decomposing an order-$d$ tensor $\ten{X}\in\R^{\vec{n}}$ by the RtL-HT algorithm for a given dimension tree $\set{T}$ and hierarchical rank $\vec{r}\in\R^{z}$ in terms of both flops and storage demand. 

Firstly, we assume that $r_{\set{\footnotesize S}}=\text{rank}(\mat{X}^{(\set{\footnotesize S})})$ for every node $\set{S}\in\set{T}$ and no truncation is performed. We assume to take as input an oversampling parameter, $p\in\N$, and  a sampling vector $\vec{w}\in\N^z$, such that $w_k$ corresponds to the number of columns to sample from $\mat{X}^{(\set{\footnotesize S})}$ where $\set{S}\in\set{T}$ is the $k$-th node. As for the hierarchical rank, to simplify the exposition, we denote by ${w}_\set{S}$ and $w_k$ the sampling value associated with the internal node $\set{S}$ and the leaf $\{k\}$, respectively.  Additionally, we define the \textcolor{black}{leaf} sampling vector as $\vec{w}_{\set{\footnotesize L}}:= (w_1, \dots, w_d)$.

In the deterministic RtL-HT the first step is computing the Tucker factors via the T-HOSVD of the input tensor at multilinear rank $\vec{r}_{\set{\tiny L}}$, and \textcolor{black}{associating} each of them to the corresponding leaf of $\set{T}$. 
In our new algorithm, we replace this step by a call to the Sub-R-HOSVD of the input tensor at multilinear rank $\vec{r}_{\set{\tiny L}}$ with sampling $\vec{w}_{\set{\tiny L}}$ and oversampling $p$. Thus, as described in section~\ref{sec:sr-HOSVD}, for each mode $k$, we form $\mat{Y}_k:=\mat{X}^{(k)}\mat{E}_k\in\R^{n_k\times w_k}$ where $\mat{E}_k\in\R^{n_{\neq k}\times w_k}$ contains on its columns $w_k$ randomly chosen vectors of the canonical basis of $\R^{n_{\neq k}}$, and then use the randomize\textcolor{black}{d} range-finder algorithm to approximate $\text{Range}(\mat{Y}_k)$. Notice that the core tensor is not formed, to reduce the computational complexity.

Once the leaf factors are computed, we visit all the internal nodes of $\set{T}$ going toward the root. 
For each node $\set{S}$ at level $z-1$, we form $\mat{Y}_{\set{\footnotesize S}} := \mat{X}^{(\set{\footnotesize S})}\mat{E}_{\set{\footnotesize S}}$ of size $(n_{\set{\footnotesize S}} \times w_{\set{\footnotesize S}})$ where $\mat{E}_{\set{\footnotesize S}}$ has $w_{\set{\footnotesize S}}$ columns corresponding to $w_{\set{\footnotesize S}}$ randomly chosen vectors of the canonical basis of $\R^{n_{\smallnotin{\set{\footnotesize {S}}}}}$. The construction of $\mat{Y}_{\set{\footnotesize S}}$ and the approximation of its range follow the same strategy described in section~\ref{sec:sr-HOSVD}. After forming the $(n_{\set{\footnotesize S}} \times r_{\set{\footnotesize S}})$ matrix $\mat{U}_{\set{\footnotesize S}}$, whose columns form an orthogonal basis approximating $\text{Range}(\mat{X}^{(\set{\footnotesize S})})$, the final steps of the node $\set{S}$ visit are equal to those of the deterministic RtL-HT. We compute the transfer tensor $\ten{B}_{\set{\footnotesize S}}$ 
following Equation~\eqref{eq:transften:comp}. 
Once all the nodes at level $z-1$ have been visited in this manner, we repeat the same visit steps for all the nodes of level $z-2$. The algorithm terminates once the dimension tree root is reached, and the transfer tensor to the root, $\ten{B}_{\set{\footnotesize R}}$, is computed as in Equation~\eqref{eq:transftenR:comp}. This novel algorithm, called {Sub-R-RtL-HT}, is summarized in  Algorithm~\ref{alg:Sub-R-HT}.

\textcolor{black}{The computational cost of the leaf factors} in the {Sub-R-RtL-HT} algorithm corresponds to the computational cost of the {Sub-R-HOSVD}\textcolor{black}{,} that is\textcolor{black}{,}
 $\mathcal{O}(\sum_{k=1}^d n_kw_kr_k+n_kr_k^2)$.   
To this term, we must add the computational complexity of the $d-2$ internal \textcolor{black}{node} SVDs, that is\textcolor{black}{,}
 $   \mathcal{O}\Bigl(\sum_{\set{\tiny S} \in(\set{T}-\set{L})}n_{\set{\tiny S}}w_{\set{\tiny S}}r_{\set{\tiny S}}  + n_{\set{\tiny S}}r_{\set{\tiny S}}\Bigr)$. 
 Assuming all the tensor modes are equal to $n$, all \textcolor{black}{node} sampling values and hierarchical rank values are equal to $w$ and $r$ respectively, then the overall computational complexity results in $\mathcal{O}\Bigl(\sum_{k=2}^{\lceil{d/2}\rceil}n^kwr  + n^kr\Bigr)$, \textcolor{black}{significantly reducing} the computational costs. 
 The main memory usage of Algorithm~\ref{alg:Sub-R-HT} is given by the storage of the matrix $\mat{Y}_k$ that amounts to $w_k$ vectors of length $n_k$ for the leaves, that is way less than $\prod_{i=1}^dn_i$ for $w_k\ll n_{\neq k}$ for its \textcolor{black}{corresponding deterministic} algorithm. Similarly, for the internal node $\set{S}$, we store $\mat{Y}_{\set{\tiny S}}$ that amounts to $w_{\set{\tiny S}}$ vectors of length $n_{\set{\tiny S}}$, while in its deterministic corresponding algorithm we store $n_{ \smallnotin \set{\tiny S}}$ columns of the same length. 

\begin{algorithm}[t!]
	\centering
	\caption{$[\{\mat{U}_k\}\,, \{\ten{B}_{\set{\footnotesize S}}\}]$ = Sub-R-RtL-HT($\ten{X}$, $\set{T}$, $\vec{r}$, $\vec{w}$, $p$)}\label{alg:Sub-R-HT}
	\begin{algorithmic}[1]
		\Input{$\ten{X}$ an order-$d$ tensor of size $\vec{n}$, a dimension tree $\set{T}$, a prescribed hierarchical rank $\vec{r}$, a sampling vector $\vec{w}$ and an oversampling parameter $p$}
		\Output{the \textcolor{black}{leaf} factor matrices $\{\mat{U}_k\}$, and the internal \textcolor{black}{node} transfer tensor $\{\ten{B}_{\set{\footnotesize S}}\}$}
        \State Initialize $z$ as the depth of the dimension tree $\set{T}$, and $\set{L}$ as the leaves of $\set{T}$.
        \State $\{\mat{U}_k\}$ = subr-HOSVD($\ten{X}$, $\vec{r}_{\set{\tiny L}}$, $p$, $\vec{w}_{\set{\footnotesize L}}$) 
		\For{$h = z-1, \ldots, 1$}
        \For{$\set{S}$ internal node at level $h$}
        \State Randomly pick $w_{\set{\footnotesize S}}$ fibers of $\ten{X}$ along modes in $\set{S}$ and store them in $\mat{Y}_{\set{\footnotesize S}}$
        \Statex{\hspace{3em}\(\triangleright\) $\mat{Y}_{\set{\footnotesize S}}$ has size $(n_{\set{\tiny S}} \times w_{\set{\tiny S}})$ where $\set{S}_{\ell}$ and $\set{S}_r$ are the children of $\set{S}$}
        \State Generate a Gaussian sketch matrix $\mat{\Omega}_k$ of size $({w_{\set{\footnotesize S}} \times (r_{\set{\footnotesize S}}+p)})$
		\State Compute an orthonormal basis {\color{black}$\mat{Q}_{\set{\footnotesize S}}$} for $\text{Range}(\mat{Y}_{\set{\footnotesize S}} \mat{\Omega}_{\set{\footnotesize S}})$
        \State {\color{black} Compute $\mat{W}_{\set{\footnotesize S}}$ storing the dominant $r_{\set{\footnotesize S}}$ left singular vectors of $\mat{Q}_{\set{\footnotesize S}}^T\mat{Y}_{\set{\footnotesize S}}$
        }
        \State {\color{black} Set $\mat{U}_{\set{\footnotesize S}}=\mat{Q}_{\set{\footnotesize S}}\mat{W}_{\set{\footnotesize S}}$}
        \State Compute the transfer tensor $\ten{B}_{\set{S}}$ between the basis of node $\set{S}$ and the bases of its children
        \[
            \ten{B}_{\set{\footnotesize S}}(i,j,k) = \langle \mat{U}_\set{S}(\cdot, i),\mat{U}_{\set{S}_\ell}(\cdot, j)\kron \mat{U}_{\set{S}_r}(\cdot, k)\rangle
        \]
        \EndFor
		\EndFor
        \State Compute the transfer tensor $\ten{B}_{\set{R}}$ between the basis of root $\set{R}$ and the bases of its children
        \[
            \ten{B}_{\set{\footnotesize R}}(1,j,k) = \langle \text{vec}(\ten{X}),\mat{U}_{\set{R}_\ell}(\cdot, j)\kron \mat{U}_{\set{R}_r}(\cdot, k)\rangle
        \]
	\end{algorithmic}
\end{algorithm}


{\color{black}
\begin{remark}
{\rm
Our randomization-based approach (fiber sampling plus range-finder) can be easily adopted within LtR-HT as well. However, the computational gains coming from this strategy would be much more modest than those achieved in \textcolor{black}{the} case of RtL-HT. Indeed, similar to ST-HOSVD, LtR-HT sees a dimension reduction once all the matrices $\mat{U}_\set{S}$ at level $h$, $h=z-1,\ldots,1$, have been computed. Further reducing small\textcolor{black}{-}dimensional objects by randomization leads to limited benefits. Moreover,
 LtR-HT presents a more sequential nature compared to RtL-HT. Indeed, in the former, only the SVDs at the same level can be potentially computed in parallel whereas in the latter we can compute all the $2(d-1)$ necessary SVDs at the same time. This feature is the main reason why we decided to work on RtL-HT rather than LtR-HT.
}\end{remark}
}

\subsection{Numerical experiments}\label{Numerical experiments - HT}
In this section, we report the numerical performance of {Sub-R-RtL-HT}, RtL-HT, and {LtR-HT} on artificially constructed problems in a serial (section~\ref{sec:ht:seqexp}) and parallel (section~\ref{sec:ht:parexp}) computing environment.
Similarly to the standard Tucker case, we assess the performance of the method on synthetic tensors while varying the tensor order $d$.
{\color{black} Further numerical results can be found in the supplementary material.}
\subsubsection{Serial tests}\label{sec:ht:seqexp}

\textcolor{black}{The tensor considered in this example is taken from~\cite[Section 6.1]{Grasedyck2009} and it is constructed as follows.
We fix the mode dimension to $n=15$ and consider tensor orders ranging from $d=4$ to $d=8$. For each tensor order, the entries of the tensor are defined as
\[
\ten{X}(i_1,\ldots,i_d)
=
\left(\sum_{j=1}^d i_j^2\right)^{-1/2}.
\]
Equivalently, the tensor can be interpreted as the evaluation of the function
$f(\vec{x}) = \norm{\vec{x}}_F^{-1}$
on a uniform grid over $[1,15]^d$.\\
The target rank is chosen according to the standard logarithmic dependence on the desired accuracy, namely $k \sim \log(1/\varepsilon)$; see~\cite{Grasedyck2009}. In particular, setting $\varepsilon=10^{-3}$ results in $k=7$.
We compare our method with the deterministic RtL-HT and LtR-HT methods.
The results are reported in Figure~\ref{fig:HT:norm}.
For each tensor order, boxplots summarize the distribution of approximation errors (Figure~\ref{fig:HT:norm:medianerr:BOX}) while logarithmic-scale plots report the median performance as a function of $d$ in order to highlight the scaling behavior of the methods in terms of accuracy (Figure~\ref{fig:HT:norm:medianerr}) and running time (Figure~\ref{fig:HT:norm:mediantime}). \\
The results achieved by {Sub-R-RtL-HT} have been obtained by considering different sampling factors. In particular, the parameter reported in the legend of Figure~\ref{fig:HT:norm} for {Sub-R-RtL-HT} denotes a scaling coefficient, which controls the number of fibers employed at each node of the dimension tree. More precisely, for every unfolding associated with a given node $\set{S}$, the algorithm selects $\min\{\alpha^{d-2} \cdot n_{\set{S}},\, n_{\smallnotin\set{S}}\}$ fibers, where $n_{\set{S}}$ is the number of rows of the corresponding unfolding and $n_{\smallnotin\set{S}}$ is the total number of available fibers. We use the scaling factor $\alpha^{d-2}$ so that the sampling size increases with the tensor order. We stress, once again, that tensor unfoldings are never explicitly formed in the implementation; the description above is provided only for clarity and to give an intuitive interpretation of the sampling rule. \\
The observed trends are consistent with those obtained in the Tucker setting. {Sub-R-RtL-HT} achieves satisfactory approximation accuracy while attaining computational timings significantly smaller than those obtained by its deterministic counterpart, namely RtL-HT. This speed-up is mainly due to the subsampling strategy, which avoids the construction of large intermediate objects and allows the algorithm to operate on substantially smaller representations. It is interesting to notice how {Sub-R-RtL-HT} is rather competitive with {LtR-HT}. The latter method is often the method of choice for computing the H-Tucker decomposition. Indeed, thanks to the dimension reductions taking place every time the left singular vectors of the matricization at a given node are available, this method is rather fast. However, its serial nature is clearly not suited for parallelization. On the other hand, as we will show in the next section, {Sub-R-RtL-HT} is easily parallelizable while also achieving satisfactory results in a serial computing environment.}

\begin{figure}[t!]
\centering
\subfloat[Median relative error \textcolor{black}{on the y-axis}.]{\label{fig:HT:norm:medianerr}
    \includegraphics[width=0.4\textwidth]
    {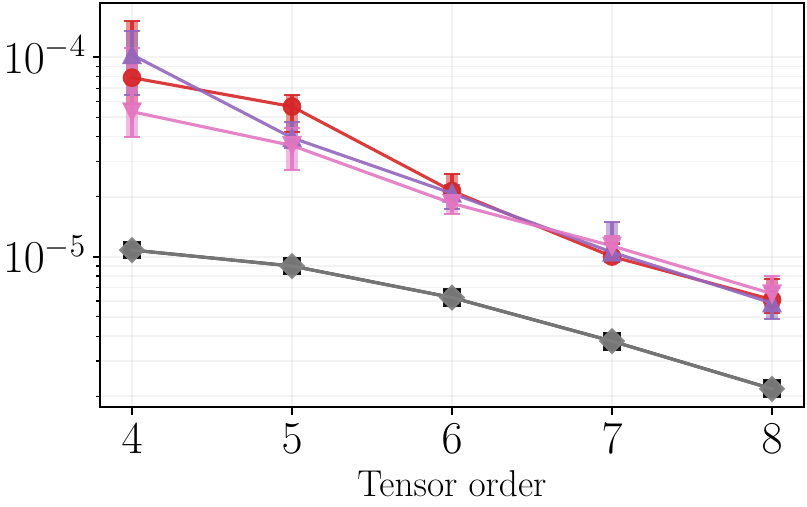}}
\hspace{15pt}
\subfloat[Median computational time \textcolor{black}{on the y-axis}.]{\label{fig:HT:norm:mediantime}
    \includegraphics[width=0.4\textwidth]
    {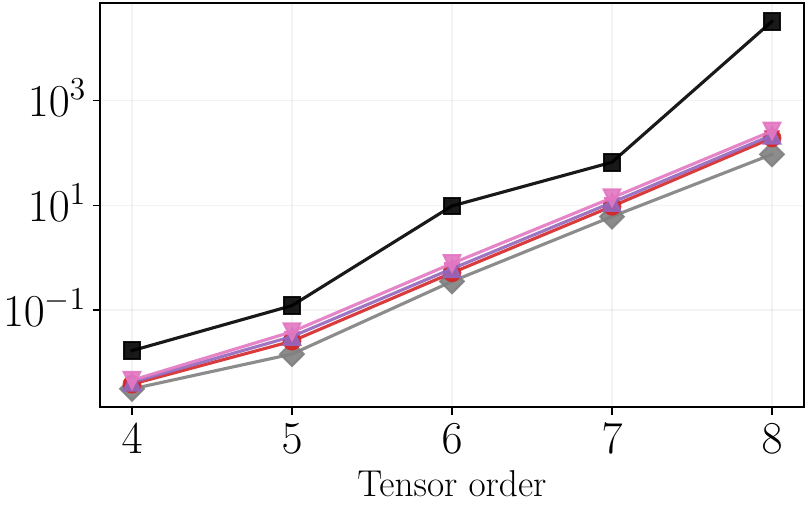}}
\vspace{5pt}
\subfloat[Boxplot of the relative errors for each order.]{\label{fig:HT:norm:medianerr:BOX}
    \includegraphics[width=0.99\linewidth]{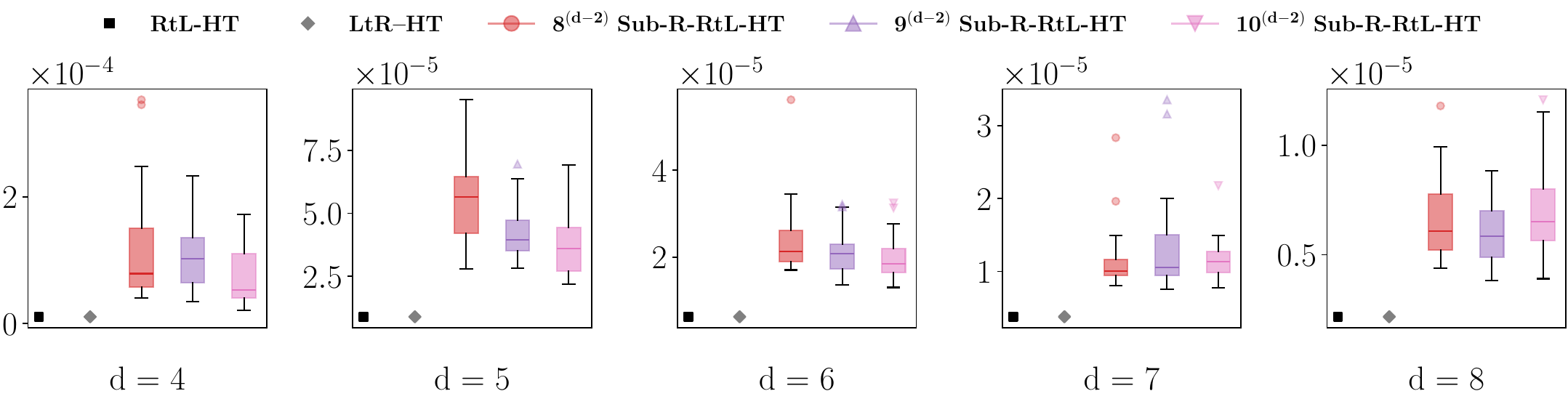}}
\caption{\color{black}Comparison of the considered methods on synthetic tensors with entries from the evaluation of the function $f(\vec{x}) = \norm{\vec{x}}_F^{-1}$ on a uniform grid over $[1, 15]^d$. All the randomized algorithms are tested over $25$ independent runs.}
\label{fig:HT:norm}
\end{figure}

\subsubsection{Parallel tests for {Sub-R-RtL-HT}}\label{sec:ht:parexp}

In this section we assess the parallel scalability of {Sub-R-RtL-HT}. 
We test the parallel implementation on low-rank tensors in the \textcolor{black}{hierarchical sense} of fixed order $d=8$ and increasing mode dimension from $n=16$ to $n=20$, with a fixed hierarchical rank $r=6$ in all cases, which is also the truncation rank used in the decomposition.
All tensors are randomly generated in the same way as described in Section~\ref{sec:ht:seqexp}, and at each node of the dimension tree we extract the same number of fibers, equal to $20n$ for each tensor.

One of the main advantage\textcolor{black}{s} of the RtL approach over LtR is the complete independence of the  SVD computations associated with the nodes of the dimension tree. This means that these SVDs can all be performed in parallel, at the same time. However, as we will show in this section, it is important to keep in mind that these SVDs do not all have the same cost due to the different dimensions of the objects we handle.  A naive distribution of the workflow may result in processes waiting \textcolor{black}{a long time} for others to complete their tasks, thus reducing the overall speed\textcolor{black}{-}up of the algorithm. Maintaining a balanced workflow among all the available processes is vital, yet challenging, to make the most out of a parallel implementation. The main goal of this section is to show the potential of our randomized approach but, at the same time, to illustrate some of the limitations intrinsic to the H-Tucker decomposition.

For a tensor of order $d=8$ and for the dimension tree shown in Figure~\ref{fig:HT:dimtree}, the algorithm requires a total of $14$ SVD computations and $7$ transfer tensors. In all our tests, we first compute the matrices $\mat{U}_\set{S}$, for all the nodes of the \textcolor{black}{tree}, in a distributed fashion. Once we have these matrices, we compute the transfer tensors $\ten{B}_i$, $i=1,\ldots,6$, again in parallel. At the end, the master process computes $\ten{B}_0$. These are our computational choices which led to satisfactory results, as we will see shortly. However, different, possibly better, distributions of the workflow can be explored as well.

We start by testing the scalability of our approach for $1$, $2$, $4$, and $8$ processes and in Figure~\ref{fig:HT:PAR} we report the results. The first thing to observe by looking at Figure~\ref{fig:HT:PAR:time} is the cost of the serial portion of the algorithm, which turns out to be much larger than in the case of Tucker (cf. Figure~\ref{fig:Tuck:PAR:time}). The main serial cost of our Sub-R-RtL-HT is the construction of the root transfer tensor. 
In principle, this step only requires the left singular vectors of the children \textcolor{black}{of the root} to be completed, but in the current implementation it is performed after all SVD computations have been completed, resulting in an additional serial section at the end of the parallel phase.
Nevertheless, {Figure~\ref{fig:HT:PAR:speedup}} showcases good scalability speed-ups which, as in the Tucker case, deteriorate as $n$ increases.
To mitigate this effect, we parallelize the fiber extraction step, following the same strategy used for {Sub-R-HOSVD}. 
In this setting, each process performs the extraction using $20$ workers. 
Although this introduces a mild oversubscription, since the total number of requested workers exceeds the number of available cores, the overall performance improves. 
This is confirmed in Figure~\ref{fig:HT:PAR_parextraction}, where we report the scaling for all process counts highlighting the improved speed-up.

To conclude, we would like to comment on some limitations of the parallelization of the RtL-HT algorithm. As we mentioned earlier, maintaining a balanced workflow is key to fully \textcolor{black}{capitalizing} on the parallel computing architecture at hand.
To show the importance of this aspect we now test the parallel scalability using $1$, $2$, $4$, $7$, $8$, and $14$ processes, with a parallel extraction of the fiber. The results are reported in Figure~\ref{fig:HT:PAR_parextraction_1_to_14}. From both the total execution timings (Figure~\ref{fig:HT:PAR_parextraction_1_to_14:time}) and the speed-ups (Figure~\ref{fig:HT:PAR_parextraction_1_to_14:speedup}) we can readily notice how passing from $4$ to $7$ processes or $8$ to $14$ is not really beneficial but rather a \textcolor{black}{waste} of computational \textcolor{black}{resources}.
This behavior is due to the unbalanced cost of the Sub-RSVD computations across the nodes of the tree, with the most expensive operations occurring at the $8$ leaf nodes. 
Moreover, as the dimension $n$ increases, the scalability worsens, similarly to what was observed for {Sub-R-HOSVD}. 
This effect is mainly due to the slower memory access caused by the simultaneous reading of large data by multiple processes.
We believe that these results show that, even though RtL is amenable \textcolor{black}{to} a full parallelization of the algorithm -- with 14 processes we were able to compute all the 14 necessary SVDs in parallel -- this does not necessarily lead to good performance and a more careful, less wasteful  implementation can achieve similar results.

\begin{figure}[t!]
\centering
\subfloat[Strong scaling as a function of the number of processes.]{\label{fig:HT:PAR:time}
\includegraphics[width=0.99\linewidth]{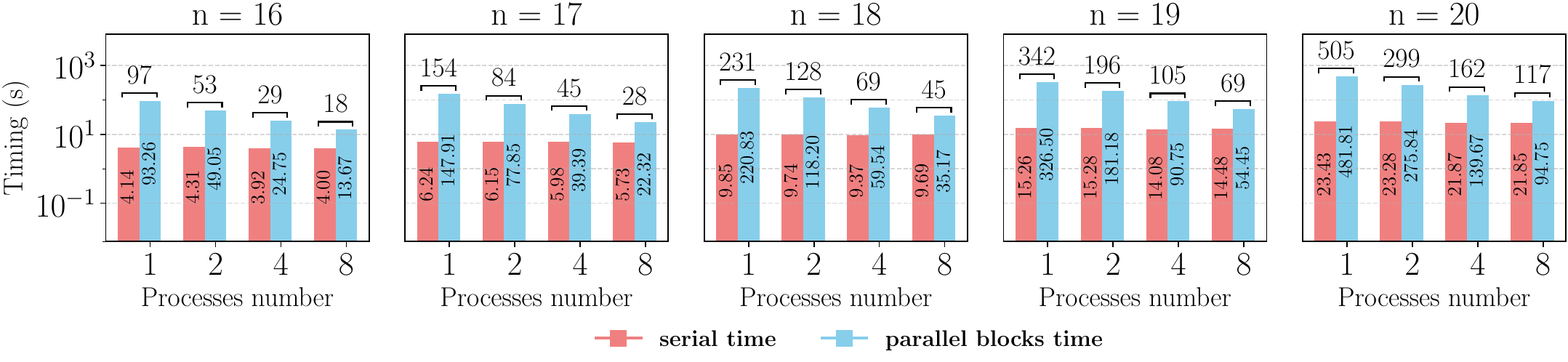}} \\
\subfloat[Speed-up  as a function of the number of processes.]{\label{fig:HT:PAR:speedup}
\includegraphics[width=0.99\linewidth]{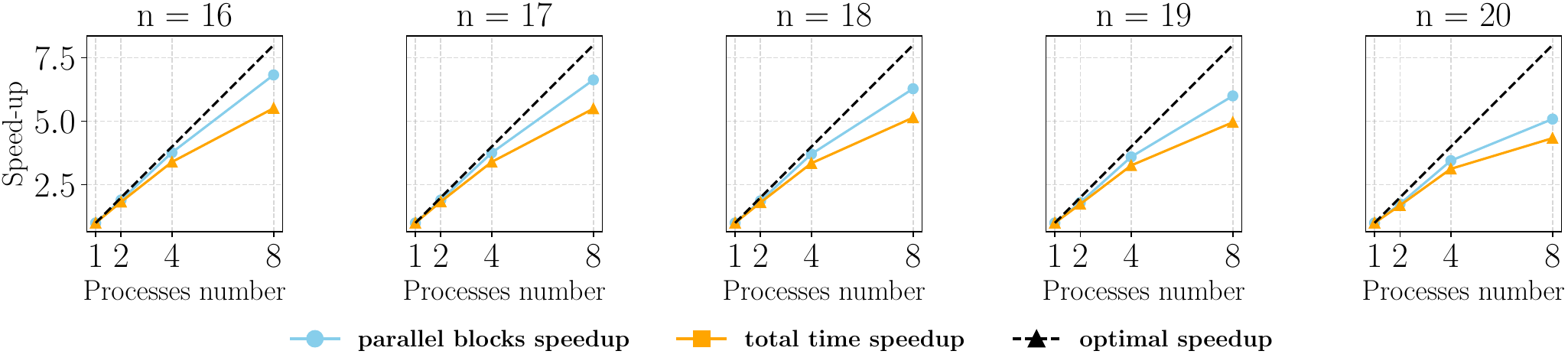}}
\caption{{Parallel performance of Sub-R-\textcolor{black}{RtL}-HT using $1$, $2$, $4$ and $8$ processes. The dashed line with slope~1 denotes the ideal linear speed-up.}}
\label{fig:HT:PAR}
\end{figure}

\begin{figure}[t!]
\centering
\subfloat[Strong scaling as a function of the number of processes.]{\label{fig:HT:PAR_parextraction:time}
\includegraphics[width=0.99\linewidth]{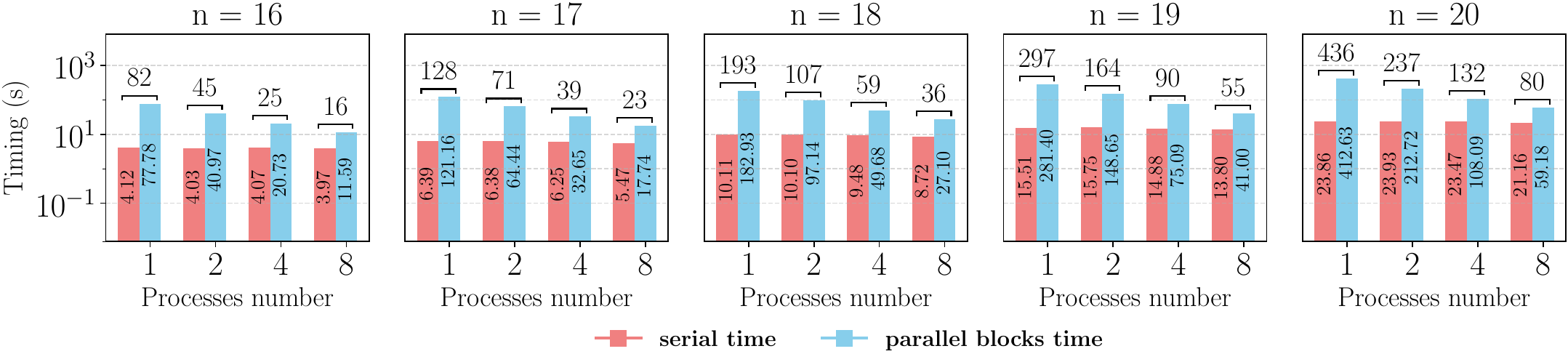}} \\


\subfloat[Speed-up  as a function of the number of processes.]{\label{fig:HT:PAR_parextraction:speedup}
\includegraphics[width=0.99\linewidth]{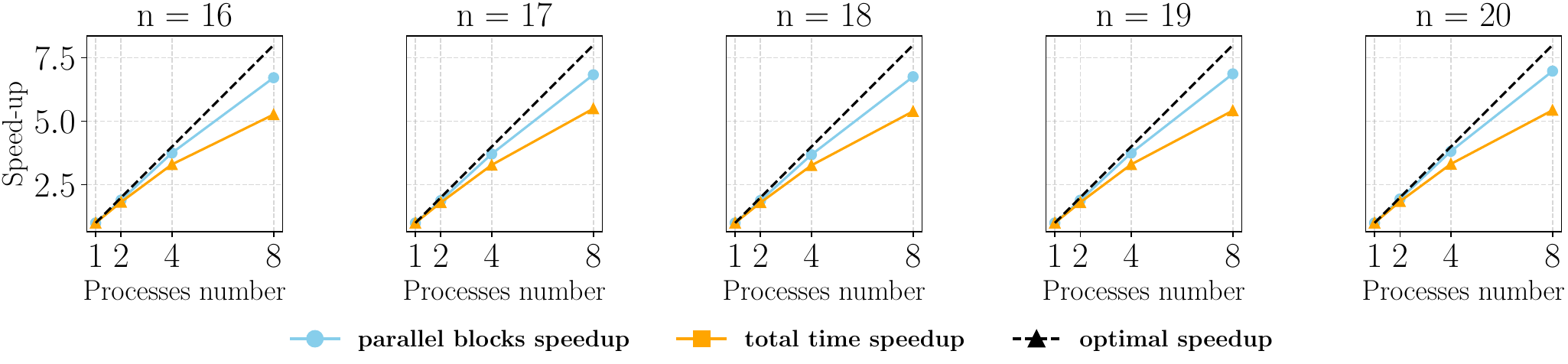}}
\caption{{Parallel performance of Sub-R-\textcolor{black}{RtL}-HT using $1$, $2$, $4$ and $8$ processes and parallel extraction. The dashed line with slope~1 denotes the ideal linear speed-up.}}
\label{fig:HT:PAR_parextraction}
\end{figure}
}{
\centering{\color{black} 
}
}

\begin{figure}[t!]
\centering
\subfloat[Strong scaling as a function of the number of processes.]{\label{fig:HT:PAR_parextraction_1_to_14:time}
\includegraphics[width=0.99\linewidth]{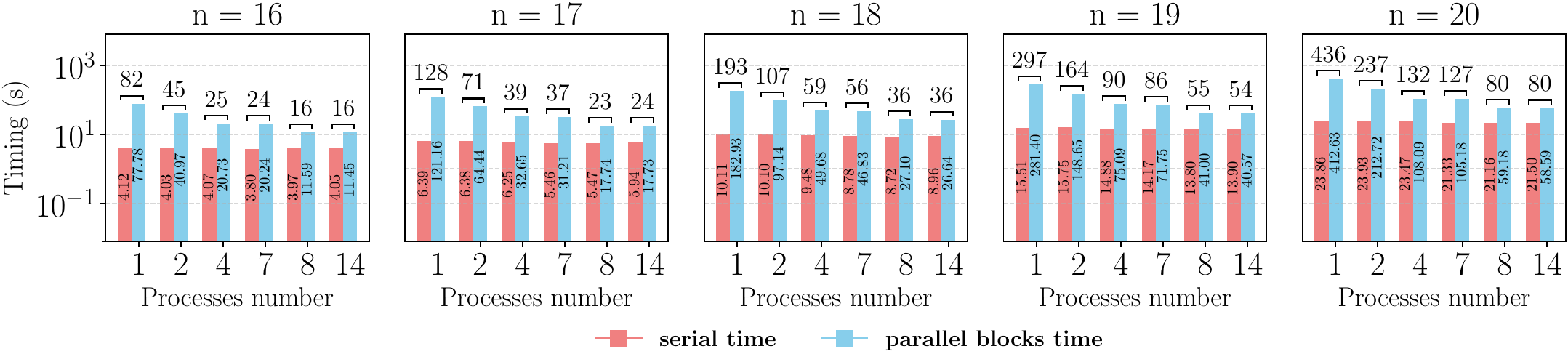}} \\
\subfloat[Speed-up  as a function of the number of processes.]{\label{fig:HT:PAR_parextraction_1_to_14:speedup}
\includegraphics[width=0.99\linewidth]{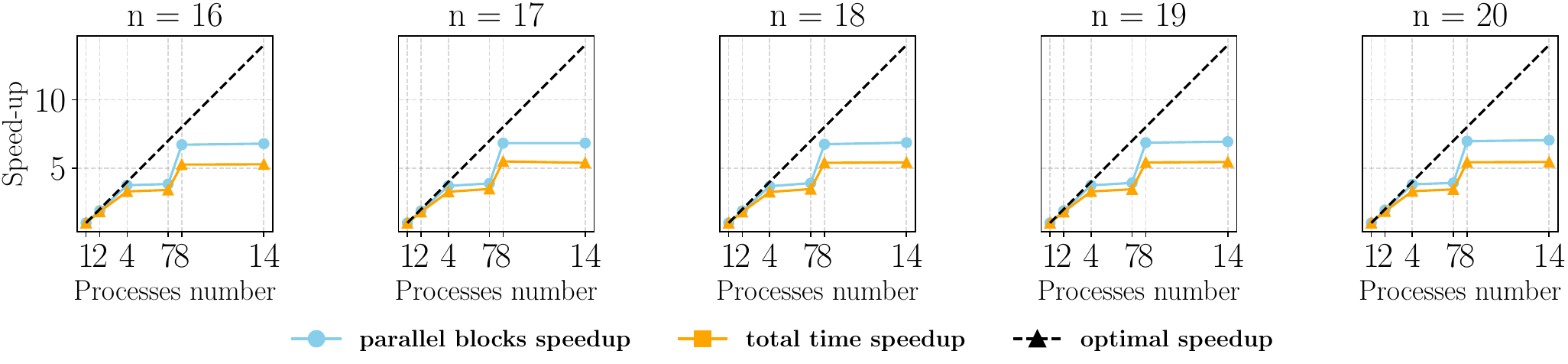}}
\caption{{Parallel performance of Sub-R-\textcolor{black}{RtL}-HT using $1$, $2$, $4$, $7$, $8$, and $14$ processes and parallel extraction. The dashed line with slope~1 denotes the ideal linear speed-up.}}
\label{fig:HT:PAR_parextraction_1_to_14}
\end{figure}

    \section{Conclusions}\label{Conclusions}
    A novel mode-parallel implementation of the Tucker decomposition has been proposed. This has been achieved by combining a sampling step along with a range-finding strategy. In particular, the sampling step is crucial to prevent the expensive construction of any matricization of the tensor thus leading to dramatic savings in both the computational costs and memory requirements of our tensor compression scheme. It is thank to this significant reduction in storage allocation that we are actually able to run our scheme in a mode-parallel fashion. Indeed, we no longer need to store $d$ copies of the full tensor as required by a naive mode-parallel implementation of the Tucker decomposition. In addition to the sampling step, the range-finder improves the quality of the computed approximation. 
    We also derived error bounds attained by our novel numerical scheme. Due to the randomized nature of the tools employed in our procedure (sampling and range-finder), these error bounds are in expected value. A panel of different numerical experiments, both in a serial and parallel computing setting, showcases the performance of our novel approach. In particular, our method turns out to be one order of magnitude faster than state-of-the-art randomization-based techniques while always achieving very similar accuracy records. Strong scaling results also illustrate the good scalability of our approach. 
    We generalized the scheme to the computation of the H-Tucker decomposition for which promising results have been reported.

We believe this work opens to a couple of very interesting research questions. The first one would be to get a better understanding of the phenomenon observed in, e.g., Figure~\ref{fig:Tuck:37:medianerr}, namely the error attained by our Sub-R-HOSVD scales better with respect to the number of modes $d$ than what happens for the (deterministic) HOSVD. As already mentioned, we believe this is due to the handling of much smaller objects in our setting which may alleviate error accumulation in finite arithmetic. It would be thus interesting to explore a finite precision analysis of our scheme. Secondly, a thorough study of the H-Tucker decomposition equipped with RSVD is still lacking in the literature. The derivation of error bounds for this scheme could be seen as a first step towards getting bounds for our Sub-R-RtL-HT.

\section*{Acknowledgments}
The first, second, and last authors are members of the INdAM Research
Group GNCS. Moreover, their work
was partially supported by the European Union - NextGenerationEU under the National Recovery and Resilience Plan (PNRR) - Mission 4 Education and research
- Component 2 From research to business - Investment 1.1 Notice Prin 2022 - DD N. 104 of 2/2/2022,
entitled “Low-rank Structures and Numerical Methods in Matrix and Tensor Computations and their
Application”, code 20227PCCKZ – CUP J53D23003620006.
%
\bibliographystyle{siamplain}
\bibliography{references}

	
\end{document}